\documentclass[12pt, reqno]{amsart}
\usepackage{amsmath, amsthm, amscd, mathrsfs, amsfonts, amssymb, graphicx, color, float}
\usepackage[utf8]{inputenc}
\usepackage[bookmarksnumbered, colorlinks, plainpages]{hyperref}
\hypersetup{colorlinks=true,linkcolor=red, anchorcolor=green, citecolor=cyan, urlcolor=red, filecolor=magenta, pdftoolbar=true}

\newtheorem{theorem}{Theorem}[section]
\newtheorem{lemma}[theorem]{Lemma}
\newtheorem{proposition}[theorem]{Proposition}
\newtheorem{corollary}[theorem]{Corollary}
\theoremstyle{definition}
\newtheorem{definition}[theorem]{Definition}
\newtheorem{example}[theorem]{Example}

\theoremstyle{remark}
\newtheorem{remark}[theorem]{Remark}

\numberwithin{equation}{section}

\DeclareMathOperator{\ran}{ran}
\DeclareMathOperator{\sspan}{span}

\begin{document}
	
	\title[Conjugations on Hilbert $C^*$-modules]{Conjugations on Hilbert $C^*$-modules}
	
	\author[M. Faregh]{Mahboubeh Faregh$^1$}
	\address{$^1$Department of Pure Mathematics, Faculty of Mathematical Sciences, Ferdowsi University of Mashhad, P. O. Box 1159, Mashhad 91775, Iran}
	\email{mahboubfaregh@yahoo.com}
	
	\author[R. Eskandari]{Rasoul Eskandari$^2$}
	\address{$^2$Department of Mathematics Education, Farhangian University, P.O. Box 14665-889, Tehran, Iran.}
	\email{Rasoul.eskandari@cfu.ac.ir, eskandarirasoul@yahoo.com}
	
	\author[M. S. Moslehian]{Mohammad Sal Moslehian$^{3*}$}
	\address{$^3$Department of Pure Mathematics, Faculty of Mathematical Sciences, Ferdowsi University of Mashhad, P. O. Box 1159, Mashhad 91775, Iran}
	\email{moslehian@um.ac.ir; msmoslehian@gmail.com}

	\subjclass{46L08, 46L05, 47A05, 47B15, 47B20.}
	\keywords{Hilbert $C^*$-module; $*$-conjugate-automorphism; real $C^*$-algebra; conjugation; polar decomposition.}
	
	\begin{abstract} 
		
		This work initiates a systematic development of conjugation theory within the framework of Hilbert $C^*$-modules, extending classical complex symmetric operator theory beyond the Hilbert space setting. We first explore the foundational structure of $*$-conjugate-automorphisms $\sharp$ on $C^*$-algebras and characterize their existence through isomorphisms with opposite algebras. We show that the set $\mathbb{L}_c(\mathscr{E})$ of all adjointable conjugate-linear operators on a Hilbert $C^*$-module $\mathscr{E}$ possesses the structure of a Hilbert $\mathbb{L}(\mathscr{E})$-module. We also introduce the key notion of $\sharp$-conjugation on a Hilbert $C^*$-module and fully characterize such modules in terms of real forms and real $C^*$-algebras. Moreover, we investigate the geometric properties of $\sharp$-conjugate duals. A central achievement is the development of polar decomposition theory for semiregular linear (conjugate-linear) operators acting on a Hilbert $C^*$-module equipped with a $\sharp$-conjugation $C$, thereby extending the celebrated Garcia--Putinar theorem [Trans. Amer. Math. Soc. 359 (2007), 3913--3931]. As a consequence, we show that any $C$-symmetric unitary operator factors as a product of two $\sharp$-conjugations. A comparative analysis of alternative definitions shows that requiring the naive adjoint condition $\langle Cx,y\rangle = \langle x,Cy\rangle^*$ would force the underlying $C^*$-algebra to be essentially commutative, thus justifying the adopted approach. We supply various examples and counterexamples that extensively illustrate our results.
	\end{abstract}
	
	\maketitle

	\section{Introduction}\label{sec1}
	
	Let $\mathscr{H}$ be a Hilbert space with inner product $\langle \cdot, \cdot \rangle$ and let $C:\mathscr{H}\to \mathscr{H}$ be a \emph{conjugation}, that is, a conjugate-linear operator that is involutive ($C^2 = I_\mathscr{H}$, the identity operator on $\mathscr{H}$) and isometric ($\langle Cx, Cy \rangle = \langle y, x \rangle$ for all $x, y \in \mathscr{H}$). An operator $T \in \mathbb{B}(\mathscr{H})$ is called \emph{complex symmetric} or \emph{$C$-symmetric} (\emph{complex skew-symmetric}, respectively) with respect to a conjugation $C$ if $T = CT^*C$ ($T = -CT^*C$, respectively). If $CT^* TC=T T^*$, then $T$ is called \emph{$C$-normal}. For instance, usual normal operators, truncated Toeplitz operators, and Hankel matrices are examples of $C$-symmetric operators; see \cite{GAR}.
	
	The study of complex symmetric and skew-symmetric operators traces back to Glazman \cite{GLA}, who examined complex symmetric differential operators. In a series of papers, Garcia, Putinar, and Wogen \cite{G1, G2, G3} presented deep results on the subject. Later, Ptak, Simik, and Wicher \cite{PTA} introduced a unified framework for the study of $C$-normal operators. 
	
	Conjugations are connected to extension theory for unbounded symmetric operators and play essential roles in von Neumann algebras including linear $*$-anti-automorphisms \cite{DIX, STO}, appear in the Tomita--Takesaki modular theory for Type III factors \cite{TAK1}, and the theory of $C^*$-algebras \cite{GUO}. 
	
	The field lies at the intersection of operator theory and complex analysis with interesting works related to truncated Toeplitz operators \cite{CAM, KO1} as well as weighted shifts \cite{ZHU, BUD}. Recent years have witnessed renewed interest from non-Hermitian quantum mechanics and spectral analysis, though mathematical developments trail behind physics applications \cite{G4}. 
	
	Several significant papers on the topic concern the structure of Hilbert spaces and operators acting on them.
	
	To the best of our knowledge, this is the first systematic study of conjugations on Hilbert $C^*$-modules from the perspective developed herein. The remainder of the paper is organized as follows.
	
	Section \ref{sec2} presents an overview of the algebraic foundation of the so-called $*$-conjugate-automorphisms $\sharp$ on $C^*$-algebras $\mathscr{A}$. Theorem \ref{thm:characterization_star_conjugation} shows their equivalence to real $C^*$-algebra structures via the fixed-point algebra $\mathscr{A}_{\mathbb{R}}$. The section also relates the existence of such automorphisms to isomorphisms with the opposite algebra (Theorem \ref{thm:opposite_characterization}). In Proposition \ref{uni}, we show that every $*$-conjugate-automorphism on $\mathbb{B}(\mathscr{H})$ is implemented by a conjugate-linear unitary.
	
	In Section \ref{sec3}, we introduce the notions of a $\sharp$-adjointable conjugate-linear operator and $\sharp$-conjugation $C$. We show that the set $\mathbb{L}_c(\mathscr{E})$ of all $\sharp$-adjointable conjugate-linear operators on a Hilbert $C^*$-module $\mathscr{E}$ possesses the structure of a Hilbert $\mathbb{L}(\mathscr{E})$-module. We characterize the $\sharp$-conjugations on a Hilbert $C^*$-module by real forms and real $C^*$-algebras (Theorem \ref{thm:real-form-characterization}).  
	
	In Section \ref{sec4}, for a given $*$-conjugate-automorphism $\sharp$ on a $C^*$-algebra $\mathscr{A}$, we provide Example \ref{Esk1} showing that there exists a Hilbert $\mathscr{A}$-module which does not admit any $\sharp$-conjugation. The induced linear $*$-anti-automorphism $T \mapsto CT^*C$ is studied, and $C$-normal and $C$-symmetric operators are defined. We establish a norm characterization of $C$-hyponormal operators (Theorem \ref{thm:norm-characterization}), a result that is new even in the classical Hilbert space setting (Corollary \ref{thm:norm-characterization1}). Theorem \ref{th:main_corrected} provides a module analogue of the correspondence between symmetric bilinear forms and $C$-symmetric operators in self-dual Hilbert $W^*$-modules. At the end of this section, we investigate the geometric properties of $\sharp$-conjugate duals in context of Hilbert $W^*$-modules.
	
	Section \ref{sec5} deals with semiregular operators and conjugate-partial isometries. Theorem \ref{lemmall} establishes a polar decomposition for semiregular conjugate-linear operators on a Hilbert $C^*$-module endowed with a $\sharp$-conjugation $C$, yielding $T = CJ|T|$ with $J$ a partial isometry. Theorem \ref{thm:decomposition} extends this to semiregular linear operators, establishing the Garcia--Putinar decomposition in the setting of Hilbert $C^*$-module. As a consequence, Corollary \ref{2con} shows that every $C$-symmetric unitary is a product of two conjugations. We demonstrate that, in contrast to the Hilbert space setting, not every unitary on a Hilbert $C^*$-module is a product of two $\sharp$-conjugations.
	
	In Section \ref{sec6}, we compare the adopted definition of conjugation with an alternative natural definition based on $\langle Cx, y \rangle = \langle x, Cy \rangle^*$. Theorem \ref{th conj} shows that this alternative leads to $C(xa) = C(x)a^*$, which forces the coefficient algebra to be essentially commutative. Theorem \ref{th conj2} indicates that in the Hilbert $\mathscr{A}$-module $\mathcal{A}$, every conjugate-linear adjointable operator is self-adjoint and vanishes in noncommutative cases with trivial center, justifying the paper's chosen approach, while Example \ref{th conj3} illustrates that conjugations exist under the adopted framework even for noncommutative algebras.

	The notation and terminology used in this paper are standard and follow those in \cite{KAD2, MUR} for operator algebras and in \cite{LAN, Manu} for the theory of Hilbert $C^*$-modules. However, we recall some key notation before using them.

	\section{$*$-conjugate-automorphisms on $C^*$-algebras} \label{sec2}
	
	In this section, we present a general framework for $*$-conjugate-automorphisms on $C^*$-algebras. 
	
	\begin{definition}\label{Esk0} By a \emph{$*$-conjugate-automorphism} $\sharp$ on a $C^*$-algebra $\mathscr{A}$, we mean a map $\sharp: \mathscr{A} \to {\mathscr{A}}$ satisfying
		\begin{enumerate}
			\item $(\lambda a + b)^\sharp = \bar{\lambda} a^\sharp + b^\sharp$ for all $a,b\in\mathscr{A}, \lambda\in\mathbb{C}$;
			\item $(a^\sharp)^\sharp = a$ for all $a\in\mathscr{A}$;
			\item $(a^\sharp)^*=(a^*)^\sharp$;
			\item $(ab)^\sharp = a^\sharp b^\sharp$ for all $a,b\in\mathscr{A}$.
		\end{enumerate}
		Evidently ${\rm sp}(a^\sharp)=\overline{{\rm sp}(a)}$ for all $a\in \mathscr{A}$. Since the norm of each self-adjoint element $c\in\mathscr{A}$ is equal to its spectral radius $r(c)$, we conclude that
		$$\|a^\sharp\|^2=\|(a^\sharp)^*a^\sharp\|=r((a^\sharp)^*a^\sharp)=r((a^*a)^\sharp)=r(a^*a)=\|a^*a\|=\|a\|^2,$$ 
		proving that the map $\sharp$ is isometric. 
	\end{definition}
	
	\begin{example}\label{mox2}
		\begin{itemize}
			\item[]
			\item [(i)] Let $\mathscr{A}$ be a $C^*$-algebra and let $\sigma:\mathscr{A}\to\mathscr{A}$ be a linear involutive $*$-anti-automorphism ($\sigma(ab)=\sigma(b)\sigma(a)$). Then $a^\sharp:=\sigma(a^*)$, defined for all $a\in\mathscr{A}$, gives rise to a $*$-conjugate-automorphism on $\mathscr{A}$.
			
			\item [(ii)] Let $\mathscr{A}=\mathbb{M}_2(\mathbb{C})$ and let $J=\begin{pmatrix}0&1\\ -1&0\end{pmatrix}$. Note that $\overline{J}=J$, $J^*=J^t=-J$, and $J^2=-I_2$. Define
			$A^\sharp:=J\overline{A}J^*$ for $A\in\mathbb{M}_2(\mathbb{C})$. Then $\sharp$ is a $*$-conjugate-automorphism on $\mathbb{M}_2(\mathbb{C})$.
		\end{itemize}
	\end{example}
	
	We next investigate the natural structure of $*$-conjugate-automorphisms on a $C^*$-algebra. 
	
	A real Banach $*$-algebra $\mathscr{A}$ is called a \emph{real $C^*$-algebra} if there exists a norm $\|\cdot\|$ on its complexification $\mathscr{A}_{\mathbb{C}}$ extending the norm on $\mathscr{A}$ such that $\mathscr{A}_{\mathbb{C}}$ is a complex $C^*$-algebra and $\mathscr{A}$ is a real $*$-subalgebra closed under this norm. The reader is referred to the monographs \cite{Goodearl, LI}.
	
	\begin{proposition}\cite[Corollary 5.2.11 and Proposition 7.3.4]{LI}
		Let $\mathscr{A}$  be a real Banach $*$-algebra. The following statements are equivalent:
		\begin{enumerate}
			\item[(a)] $\mathscr{A}$ is a real $C^*$-algebra;
			\item[(b)] $\mathscr{A}$ is isometrically isomorphic to a norm closed $*$-subalgebra of $\mathbb{B}(\mathscr{H})$ for a real Hilbert space $\mathscr{H}$;
			\item[(c)] $1 + a^*a$ is invertible in $\mathscr{A}$ (if $\mathscr{A}$ is nonunital, then we consider its unitization) and $\|a^*a\| = \|a\|^2$, for all $a\in \mathscr{A}$;
			\item [(d)] The inequality $\|a^*\|\,\|a\| \leq \|a^*a+b^*b\|$ holds for all $a, b\in \mathscr{A}$;
			\item[(e)] $\mathscr{A}$ is hermitian (that is, ${\rm sp}(a) \subseteq \mathbb{R}$ for all self-adjoint elements $a \in \mathscr{A}$) and $\|a^*a\| = \|a\|^2$, for all $a\in \mathscr{A}$.
		\end{enumerate}
	\end{proposition}
	
	A real $C^*$-algebra $\mathscr{A}$ has complexification $\mathscr{A}_{\mathbb C}=\mathscr{A}\otimes_{\mathbb R}\mathbb{C}$, meaning the completion of the algebraic tensor product with respect to its canonical $C^*$-norm. Algebraically, it can be identified with $\mathscr{A}\oplus i\mathscr{A}$, with operations $(a+ib)(a'+ib')=(aa'-bb')+i(ab'+ba')$, $(a+ib)^*=a^*-ib^*$, and $(a+ib)^\sharp=a-ib$; see the comprehensive article \cite{MP} for more information.

	\begin{theorem}\label{thm:characterization_star_conjugation}
		Let $\mathscr{A}$ be a $C^*$-algebra, and let ${\rm Conj}_*(\mathscr{A})$ denote the set of all $*$-conjugate-automorphisms of $\mathscr{A}$ in the sense of Definition~\ref{Esk0}. Then the following statements are equivalent.
		\begin{enumerate}
			\item[(i)] ${\rm Conj}_*(\mathscr{A})\neq\varnothing$.
			\item[(ii)] There exists a closed real $*$-subalgebra $\mathscr{A}_\mathbb{R}$ of $\mathscr{A}$ such that $$\mathscr{A}=\mathscr{A}_\mathbb{R}\oplus i\mathscr{A}_\mathbb{R}$$ as real Banach spaces.
			\item[(iii)] There exists a real $C^*$-algebra $\mathscr{B}$ such that $\mathscr{A}$ is $*$-isomorphic, as a complex $C^*$-algebra, to the complexification $$\mathscr{B}_\mathbb{C}=\mathscr{B}\otimes_\mathbb{R}\mathbb{C}.$$
		\end{enumerate}
	\end{theorem}
	
	\begin{proof}
		(i) $\Longrightarrow$ (ii). Let $\sharp$ be a $*$-conjugate-automorphism of $\mathscr{A}$. Set $\mathscr{A}_\mathbb{R}=\{a\in\mathscr{A}:a^\sharp=a\}$. Since $\sharp$ is additive and multiplicative, $\mathscr{A}_\mathbb{R}$ is closed under addition and multiplication. It is also closed under multiplication by real scalars. Furthermore, if $x\in\mathscr{A}_\mathbb{R}$, then $(x^*)^\sharp=(x^\sharp)^*=x^*$, so $x^*\in\mathscr{A}_\mathbb{R}$. Thus, $\mathscr{A}_\mathbb{R}$ is a real $*$-subalgebra of $\mathscr{A}$.
		
		For $a\in\mathscr{A}$, define $x=(a+a^\sharp)/2$ and $y=(a-a^\sharp)/(2i)$. Using the conjugate-linearity and involutivity of $\sharp$, we obtain $x^\sharp=x$ and $y^\sharp=y$. Hence, $x,y\in\mathscr{A}_\mathbb{R}$ and $a=x+iy$.
		
		To prove uniqueness, suppose that $x+iy=0$ with $x,y\in\mathscr{A}_\mathbb{R}$. Applying $\sharp$ gives $x-iy=0$. Adding and subtracting the two equations yields $x=0$ and $y=0$. Therefore, $\mathscr{A}=\mathscr{A}_\mathbb{R}\oplus i\mathscr{A}_\mathbb{R}$ as real vector spaces.
		
		Since $\mathscr{A}_\mathbb{R}$ is a closed subspace of the Banach space $\mathscr{A}$, it is complete. Moreover, for every $x\in\mathscr{A}_\mathbb{R}$, $\|x^*x\|=\|x\|^2$. Also, $1+x^*x$ is invertible in the unitization of $\mathscr{A}_\mathbb{R}$. Indeed, it is invertible in the unitization of $\mathscr{A}$, and its inverse is fixed by the extended map $\sharp$. Therefore Proposition 2.3 implies that $\mathscr{A}_\mathbb{R}$ is a real $C^*$-algebra. In addition, $\mathscr{A}$ is canonically the complexification of $\mathscr{A}_\mathbb{R}$; see \cite[Subsection 6.3]{MP}.

		(ii) $\Longrightarrow$ (iii). Let $\mathscr{B}:=\mathscr{A}_\mathbb{R}$. The map $\Phi:\mathscr{B}\otimes_\mathbb{R}\mathbb{C}\to\mathscr{A}$ given by $\Phi(x\otimes(\lambda+i\mu))=(\lambda+i\mu)x$ can be extended to a complex $*$-isomorphism. 
		
		(iii) $\Longrightarrow$ (i). Suppose that $\mathscr{A}\cong\mathscr{B}\otimes_\mathbb{R}\mathbb{C}$ for a real $C^*$-algebra $\mathscr{B}$. The canonical $*$-conjugate-automorphism on $\mathscr{B}\otimes_\mathbb{R}\mathbb{C}$ is defined by $(b\otimes\lambda)^\sharp=b\otimes\overline{\lambda}$. Transporting this map through a complex $*$-isomorphism from $\mathscr{B}\otimes_\mathbb{R}\mathbb{C}$ onto $\mathscr{A}$ gives a $*$-conjugate-automorphism of $\mathscr{A}$.
	\end{proof}
	
	The next theorem shows that $\mathscr{A}\cong\mathscr{A}^{\mathrm{op}}$ is equivalent to the existence of a conjugate-linear $*$-automorphism, but it is not, in general, by itself a characterization of the $*$-conjugate-automorphisms in Definition~\ref{Esk0}. Furthermore, it asserts that
	\[
	\mathscr{A}\cong\mathscr{A}^{\mathrm{op}}\iff\mathscr{A}\cong\overline{\mathscr{A}}.
	\]
	
	\begin{theorem}\label{thm:opposite_characterization}
		Let $\mathscr{A}$ be a $C^*$-algebra. The following are equivalent.
		\begin{enumerate}
			\item[(i)] There exists a conjugate-linear, multiplicative, $*$-preserving bijection $\rho:\mathscr{A}\to\mathscr{A}$. (No involutivity of $\rho$ is assumed.)
			\item[(ii)] $\mathscr{A}$ is linear $*$-isomorphic to its opposite $C^*$-algebra $\mathscr{A}^{\mathrm{op}}$.
			\item[(iii)] $\mathscr{A}$ is linear $*$-isomorphic to its conjugate $C^*$-algebra $\overline{\mathscr{A}}$.
		\end{enumerate}
	\end{theorem}
	
	\begin{proof}
		If $\rho$ is as in (i), then $\Phi_\rho:\mathscr{A}\to\mathscr{A}^{\mathrm{op}}$ defined as $\Phi_\rho(a)=(\rho(a))^*$ is a linear $*$-isomorphism, and so (ii) holds. Conversely, if (ii) holds, and $\Phi:\mathscr{A}\to\mathscr{A}^{\mathrm{op}}$ is a linear $*$-isomorphism, then $\rho(a)=\Phi(a)^*$ is a conjugate-linear, multiplicative, $*$-preserving bijection of $\mathscr{A}$, and so (i) holds.
		
		The standard map 
		\[
		\overline{\mathscr{A}}\to\mathscr{A}^{\mathrm{op}},\qquad \overline a\longmapsto a^*
		\]
		is a $*$-isomorphism and shows that the equivalence of (ii) and (iii).
	\end{proof}
	
	Next, suppose that $\mathscr{A}$ is a $C^*$-algebra. Let 
	\[
	\tau:\mathscr{A}^{\mathrm{op}}\to\mathscr{A},\qquad \tau(a^{\mathrm{op}})=a^*.
	\]
	Then $\tau$ is a conjugate-linear, multiplicative, $*$-preserving bijection. It is straightforward to verify that $\mathscr{A}$ admits a $*$-conjugate-automorphism $\sharp$ if and only if there exists a linear $*$-isomorphism $\Phi:\mathscr{A}\to\mathscr{A}^{\mathrm{op}}$ such that
	$(\tau\circ\Phi)^2={\rm id}_{\mathscr{A}}$.
	
	Let us denote the group of linear $*$-automorphisms of $\mathscr{A}$ by ${\rm Aut}_*(\mathscr{A})$.
	
	\begin{theorem}\label{thm:all_star_conjugations}
		Suppose that $\mathscr{A}$ admits a $*$-conjugate-automorphism $\sharp_0$. Then
		\[
		{\rm Conj}_*(\mathscr{A})=\left\{\alpha\circ\sharp_0:\alpha\in{\rm Aut}_*(\mathscr{A}),\ \alpha\circ\sharp_0\circ\alpha\circ\sharp_0={\rm id}_{\mathscr{A}}\right\}.
		\]
	\end{theorem}
	
	\begin{proof}
		Let $\sharp\in{\rm Conj}_*(\mathscr{A})$ and define $\alpha=\sharp\circ\sharp_0$. Then, $\alpha\in{\rm Aut}(\mathscr{A})$. The reverse inclusion is clear. 
	\end{proof}
	
	Let $S:\mathscr{H}\to\mathscr{H}$ be a conjugate-linear operator on Hilbert space $\mathscr{H}$. Its adjoint $S^*$ is the conjugate-linear operator defined by $\langle Sx, y \rangle = \overline{\langle x, S^* y \rangle}$ for all $x,y \in \mathscr{H}$.
	A conjugate-linear operator $U:\mathscr{H}\to\mathscr{H}$ is called \emph{conjugate-linear unitary} if it is surjective and $\langle Ux, Uy\rangle = \langle y, x\rangle$ for all $x,y\in\mathscr{H}$. Therefore, a conjugation on $\mathscr{H}$ is a conjugate-linear unitary $C$ satisfying $C^2 = I_\mathscr{H}$.

	The following proposition may be known in the literature, however, we provide a proof for the sake of reader's convenience. Every rank-one projection on a Hilbert space $\mathscr{H}$ is of the form $P_x=x\overline{\otimes} x$ for some unit vector $x\in\mathscr{H}$. Recall that for given vectors $x,y \in \mathscr{H}$, we can define the rank-one operator $x\overline{\otimes} y$ by $(x\overline{\otimes} y)z=\langle z,y\rangle x,\,\, z\in \mathscr{H}$. 
	
	\begin{proposition}\label{uni}
		Let $\mathscr{H}$ be a complex Hilbert space and let $\sharp:\mathbb{B}(\mathscr{H})\to \mathbb{B}(\mathscr{H})$ be a $*$-conjugate-automorphism. Then there exists a conjugate-linear unitary operator $U:\mathscr{H}\to \mathscr{H}$ such that $T^\sharp=UTU^{-1}$ for all $T\in\mathbb{B}(\mathscr{H})$.
	\end{proposition}
	
	\begin{proof}
		Clearly, $(P_x^\sharp)^2=(P_x^2)^\sharp=P_x^\sharp$ and $(P_x^\sharp)^*=(P_x^*)^\sharp=P_x^\sharp$. Hence, $P_x^\sharp$ is a projection. We claim that $P_x^\sharp$ is minimal. Suppose that $Q$ is a nonzero projection satisfying $Q\leq P_x^\sharp$. Then $QP_x^\sharp=Q$. Applying $\sharp$ and using multiplicativity gives $(P_x^\sharp)^\sharp Q^\sharp=Q^\sharp$. Since $\sharp$ is involutive, $P_x Q^\sharp=Q^\sharp$. Hence, $Q^\sharp\leq P_x$. Since $P_x$ is a minimal projection in $\mathbb{B}(\mathscr{H})$, we have $Q^\sharp=P_x$, whence $Q=P_x^\sharp$. Thus, $P_x^\sharp$ is minimal. Therefore, $P_x^\sharp$ is a rank-one projection.
		
		Consequently, for every unit vector $x\in\mathscr{H}$, there exists a unit vector $z_x\in\mathscr{H}$ such that $P_x^\sharp=P_{z_x}$. Let $x,y\in\mathscr{H}$ be unit vectors. We have $\|(P_x P_y)^\sharp\| =\|P_x^\sharp P_y^\sharp\|$. Hence, $\|\langle y, x \rangle (x\overline{\otimes} y)^\sharp\|=\|\langle z_y,z_x \rangle (z_x\overline{\otimes} z_y)\|$, whence, 
		\begin{align}\label{fem1}
			|\langle x,y \rangle|^2=|\langle z_x,z_y \rangle|^2.
		\end{align}
		On the other hand, $P_xP_y=\langle y,x\rangle\, x\overline{\otimes} y$, and so 
		\begin{align}\label{fem2}
			{\rm Tr}(P_xP_y)=|\langle x,y\rangle|^2.
		\end{align}
		Similarly, $P_x^\sharp P_y^\sharp=P_{z_x}P_{z_y}=\langle z_y,z_x\rangle\, z_x\overline{\otimes} z_y$, so that
		\begin{align}\label{fem3}
			{\rm Tr}(P_x^\sharp P_y^\sharp)=|\langle z_x,z_y\rangle|^2.
		\end{align}
		
		Now, we put $\varphi(P_x):=P_x^\sharp$ for $\|x\|=1$. Since $\sharp$ is involutive, $\varphi$ is a bijection of the rank-one projections. Equalities \eqref{fem1}, \eqref{fem2}, and \eqref{fem3} show that ${\rm Tr}(\varphi(P_x)\varphi(P_y))={\rm Tr}(P_xP_y)=|\langle x,y\rangle|^2$ for all unit vectors $x,y\in\mathscr{H}$. Hence, by Wigner's theorem (see, for example, \cite{SIM}), there exists either a linear unitary or a conjugate-linear unitary operator $U:\mathscr{H}\to \mathscr{H}$ such that $P_x^\sharp=UP_xU^{-1}$ for all $\|x\|=1$.
		
		We show that $U$ must be conjugate-linear. Suppose, on the contrary, that $U$ is linear. For arbitrary $x,y\in\mathscr{H}$, the polarization identity gives
		\begin{align}\label{fem5}
			x\overline{\otimes} y=\frac{1}{4}\left(P_{x+y}-P_{x-y}+iP_{x+iy}-iP_{x-iy}\right).
		\end{align}
		Since $\sharp$ is conjugate-linear and $P_z^\sharp=UP_zU^{-1}$ for every nonzero $z$ (after normalizing $z$), we obtain
		\[
		\begin{aligned}
			(x\overline{\otimes} y)^\sharp
			&=\frac{1}{4}\left(P_{x+y}^\sharp-P_{x-y}^\sharp-iP_{x+iy}^\sharp+iP_{x-iy}^\sharp\right)\\
			&=\frac{1}{4}\left(UP_{x+y}U^{-1}-UP_{x-y}U^{-1}-iUP_{x+iy}U^{-1}+iUP_{x-iy}U^{-1}\right)\\
			&=U(x\overline{\otimes} y)U^{-1}.
		\end{aligned}
		\]
		Thus, $A^\sharp=UAU^{-1}$ for every rank-one operator $A$.
		
		Now let $A\neq 0$ be a rank-one operator and let $\lambda\in\mathbb{C}$. Since $\sharp$ is conjugate-linear, $(\lambda A)^\sharp=\overline{\lambda}A^\sharp$. Since $U$ is assumed to be linear, the preceding identity for rank-one operators also gives $(\lambda A)^\sharp=U(\lambda A)U^{-1}=\lambda UAU^{-1}=\lambda A^\sharp$. Hence, $\lambda A^\sharp=\overline{\lambda}A^\sharp$ for all $\lambda\in\mathbb{C}$. Taking $\lambda=i$ yields $iA^\sharp=-iA^\sharp$, which is impossible because $A^\sharp\neq 0$. Therefore, $U$ cannot be linear, and so $U$ is conjugate-linear.
		
		Define $\Phi:\mathbb{B}(\mathscr{H})\to \mathbb{B}(\mathscr{H})$ by $\Phi(T)=U^{-1}T^\sharp U$. Because both $\sharp$ and the map $T\longmapsto U^{-1}TU$ are conjugate-linear, $\Phi$ is complex-linear. Explicitly,
		\[
		\Phi(\lambda T)=U^{-1}(\lambda T)^\sharp U=U^{-1}(\overline{\lambda}T^\sharp)U=\lambda U^{-1}T^\sharp U=\lambda\Phi(T).
		\]
		Moreover, $\Phi$ is multiplicative, unital, $*$-preserving, and bijective. Thus, $\Phi$ is a linear $*$-automorphism of $\mathbb{B}(\mathscr{H})$.
		
		From $P_x^\sharp=UP_xU^{-1}$ we obtain $\Phi(P_x)=P_x$ for all $\|x\|=1$. By linearity and using \eqref{fem5}, $\Phi$ therefore fixes every rank-one operator. Finally, let $T\in\mathbb{B}(\mathscr{H})$ and $0\neq y\in\mathscr{H}$. Put $e=y/\|y\|$. Then $TP_e=(Te)\overline{\otimes} e$. Therefore, $\Phi(T)P_e=\Phi(T)\Phi(P_e)=\Phi(TP_e)=TP_e$. Applying both sides to $e$ gives $\Phi(T)e=Te$, whence $\Phi(T)y=Ty$. Hence, $\Phi(T)=T$. Thus, $U^{-1}T^\sharp U=T$, and consequently $T^\sharp=UTU^{-1}$ for all $T\in\mathbb{B}(\mathscr{H})$. 
	\end{proof}
	
	\section{Conjugations in the setting of Hilbert $C^*$-Modules}\label{sec3}
	
	Throughout this section, we assume that $\mathscr{A}$ is a $C^*$-algebra and $\mathscr{E}$ is a (right) Hilbert $\mathscr{A}$-module. 
	
	A map $T: \mathscr{E} \to \mathscr{E}$ is adjointable if there exists a map $T^*: \mathscr{E} \to \mathscr{E}$ such that $\langle Tx,y\rangle=\langle x,T^*y\rangle$ for all $x,y\in \mathscr{E}$. Then, both $T$ and $T^*$ are bounded $\mathscr{A}$-linear operators \cite[p. 8]{LAN}. We denote by $\mathbb{L}(\mathscr{E})$ the $C^*$-algebra of all adjointable linear operators. 
	
	\begin{definition}
		Let $\mathscr{A}$ be a $C^*$-algebra equipped with a $*$-conjugate-automorphism $\sharp$ and let $\mathscr{E}$ be a (right) Hilbert $\mathscr{A}$-module. A bounded conjugate-linear operator $S: \mathscr{E} \to \mathscr{E}$ is called \emph{$\sharp$-adjointable} if there exists a bounded conjugate-linear operator $S^*: \mathscr{E} \to \mathscr{E}$ such that
		$$\langle Sx,y\rangle=\langle x,S^*y\rangle^\sharp \qquad (x,y \in \mathscr{E}).$$
		Such an operator $S^*$ is called the $\sharp$-adjoint operator of $S$. We denote by $\mathbb{L}_c(\mathscr{E})$ the set of all bounded conjugate-linear operators $S$ on $\mathscr{E}$ that are $\sharp$-adjointable. 
	\end{definition}
	
	For every operator $S\in \mathbb{L}_c(\mathscr{E})$ we have
	\begin{align*}
		\langle S(xa),y\rangle&=\langle xa,S^*y\rangle^\sharp=\left(a^*\langle x,S^*y\rangle \right)^\sharp=(a^*)^\sharp\langle x,S^*y\rangle^\sharp=(a^\sharp)^*\langle Sx,y\rangle=\langle S(x)a^\sharp,y\rangle,
	\end{align*}
	whence
	$$S(xa)=S(x)a^\sharp, \qquad (a\in \mathscr{A}, x\in \mathscr{E}).$$
	We note that some mathematicians denote $S^*$ defined as above by $S^\sharp$. However, we use the notation $T^*$ due to it is more consistent to usual notation for adjoint of a linear operator.
	
	The usual scalar product does not work in this setting. So we define the scalar product $\lambda \cdot T$ for each $T\in \mathbb{L}_c(\mathscr{E})$ and $\lambda\in \mathbb{C}$ as follows
	\[
	\lambda \cdot T (x):=\bar{\lambda}T(x), \quad(x\in \mathscr{E}).\]
	Then 
	\[
	\langle \lambda \cdot T(x), y\rangle =\langle \bar{\lambda}Tx,y\rangle ={\lambda} \langle Tx,y\rangle={\lambda} \langle x,T^*y\rangle^\sharp=\langle x, \bar{\lambda}T^*x\rangle^\sharp=\langle x, \lambda \cdot T^*y\rangle^\sharp
	\]
	Hence, $(\lambda \cdot T)^*=\lambda \cdot T^*$.
	
	We also define the action of right $\mathbb{L}(\mathscr{E})$-module on $\mathbb{L}_c(\mathscr{E})$ by $TS$ where $T\in \mathbb{L}_c(\mathscr{E})$ and $S\in \mathbb{L}(\mathscr{E})$, which is compatible with the defined scalar product.
	
	It is straightforward to show that $\mathbb{L}_c(\mathscr{E})$ has a
	natural structure as an inner product $\mathbb{L}(\mathscr{E})$-module via the $C^*$-inner product $\langle S,T\rangle_c=S^*T$. Note that for each $x\in \mathscr{E}$, we have 
	\[
	\left((\lambda \cdot S)^*T\right)x=(\lambda \cdot S^*)Tx=\bar{\lambda}S^*(Tx)=\left(\bar{\lambda}S^*T\right)x
	\]
	So $\langle \lambda \cdot S,T\rangle_c=\bar{\lambda} \langle S,T\rangle_c$. In addition, $T^{**}=T$. Similarly, $\langle S,\lambda \cdot T\rangle_c =\lambda \langle S,T\rangle_c$.
	We can use the usual techniques in $\mathbb{B}(\mathscr{H})$ to show that for the induced norm $\|T\|_c:=\|\langle T,T\rangle_c\|^{1/2}$, 
	$\|T\|_c^2=\|T^*T\|$ and $\|T\|_c=\|T\|$, where $\|\cdot\|$ denotes the usual operator norm on the complete space of bounded conjugate linear operators on $\mathscr{E}$. 
	
	Thus, $\mathbb{L}_c(\mathscr{E})$ is a
	Hilbert $\mathbb{L}(\mathscr{E})$-module. To simplify the notation, from now on, we drop the index $_c$ and denote $\langle S,T\rangle_c$ by $\langle S,T\rangle$ and $\|T\|_c$ by $\|T\|$ without any ambiguity.
	
	Moreover, if $T, S\in \mathbb{L}_c(\mathscr{E})$, then
	$$\langle TSx,y\rangle=\langle Sx,T^*y\rangle^\sharp=\langle x,S^*T^*y\rangle^{\sharp\sharp}=\langle x,S^*T^*y\rangle,$$
	from which we conclude that $TS\in \mathbb{L}(\mathscr{E})$. Note that, unlike in the Hilbert space case, a linear operator and a conjugate-linear operator need not be adjointable and $\sharp$-adjointable, respectively.

	\begin{example}
		Let $\mathscr{A} = \mathbb{M}_2(\mathbb{C})$ be equipped with the $*$-conjugate-automorphism $\sharp: \mathscr{A} \to \mathscr{A}$ defined by entrywise complex conjugation $a^\sharp = \overline{a}$ for $a \in \mathscr{A}$. Let $\mathscr{E} = \ell^2(\mathscr{A})$ be the standard right Hilbert $C^*$-module over $\mathscr{A}$ with inner product $\langle x, y \rangle = \sum_{n=1}^{\infty} x_n^* y_n$, where $x = (x_n)_{n=1}^{\infty}$, $y = (y_n)_{n=1}^{\infty} \in \ell^2(\mathscr{A})$. Define $T: \mathscr{E} \to\mathscr{E} $ by $T(x_n) = \frac{1}{n}\,\overline{x_n}$ for $n \in \mathbb{N}$, for $x = (x_n)_{n=1}^{\infty} \in \ell^2(\mathscr{A})$. It is straightforward to verify that $T$ is $\sharp$-adjointable with $\sharp$-adjoint operator $T^*: \mathscr{E} \to \mathscr{E}$ by $T^*(y_n) = \frac{1}{n} \overline{y_n}$ for $n \in \mathbb{N}$, for $y = (y_n)_{n=1}^{\infty} \in \ell^2(\mathscr{A})$.
	\end{example}
	
	\begin{example} \label{phi} Let $x, y\in \mathscr{E}$ and define $\phi_{x,y}:\mathscr{E}\to\mathscr{E}$ by 
		\[
		\phi_{x,y}(z)=x\langle y,z\rangle^\sharp.
		\]
		We observe that $\phi_{x,y}\in\mathbb{L}_c(\mathscr{E})$, since 
		\begin{align*}
			\langle \phi_{x,y}z,w\rangle &=\langle x\langle y,z\rangle^\sharp,w\rangle=\langle z,y\rangle^\sharp\langle x,w\rangle\\
			&=\left(\langle z,y\rangle\langle x,w\rangle^\sharp\right)^\sharp=\langle z,y\langle x,w\rangle^\sharp\rangle^\sharp\qquad (z,w\in \mathscr{E}).
		\end{align*}
		This means that $\phi_{x,y}^*(w)=y\langle x,w\rangle^\sharp=\phi_{y,x}(w)$ for all $w\in \mathscr{E}$. Hence, $\phi_{x,y}^*=\phi_{y,x}$, a relation similar to corresponding rank-one operators on Hilbert spaces.
	\end{example}

	Now, we define the key notion of conjugation in the setting of Hilbert $C^*$-modules.
	
	\begin{definition}\label{ours}
		Let $\mathscr{E}$ be a Hilbert $C^*$-module over a $C^*$-algebra $\mathscr{A}$ endowed with a $*$-conjugate-automorphism. A \emph{conjugation} on $\mathscr{E}$ with respect to $\sharp$, or $\sharp$-conjugation, is a bounded involutive conjugate-linear operator $C:\mathscr{E}\to\mathscr{E}$ satisfying
		\[
		\langle Cx,Cy\rangle=\langle x,y\rangle^\sharp \qquad (x,y\in\mathscr{E}),
		\] 
		which is equivalent to $\langle Cx,y\rangle=\langle x,Cy\rangle^\sharp,\,\,x,y\in\mathscr{E}$. Also, equivalently, $C^*=C=C^{-1}$ in the sense of $\sharp$-adjointable conjugate-linear operators. Since
		$\|Cx\|^2=\|\langle Cx,Cx\rangle\|=\|\langle x,x\rangle^\sharp\|=\|\langle x,x\rangle\|=\|x\|^2$,
		the map $C$ is isometric. 
	\end{definition}
	
	\begin{proposition}
		Let $\mathscr{A}$ be a $C^*$-algebra equipped with a $*$-conjugate-automorphism
		$\sharp:\mathscr{A}\to\mathscr{A}$.
		Then $\mathscr{A}$, regarded as the standard right Hilbert $C^*$-module over itself,
		admits a $\sharp$-conjugation.
	\end{proposition}
	
	\begin{proof}
		It is enough to define $C:\mathscr{A}\to\mathscr{A}$ by $C(a):=a^\sharp$. 
	\end{proof}
	
	For a given $*$-conjugate-automorphism $\sharp$ on a $C^*$-algebra $\mathscr{A}$, the following example demonstrates that there exists a Hilbert $\mathscr{A}$-module which does not admit any $\sharp$-conjugation.
	
	\begin{example}\label{Esk1}
		Let $\mathscr{A}=\mathbb{C}\oplus\mathbb{C}$ with the natural $*$-conjugate-automorphism 
		\[
		(a,b)^\sharp=(\overline b,\overline a), \qquad (a,b)\in\mathscr{A},
		\] 
		which is $*$-preserving and isometric. Consider the right Hilbert $\mathscr{A}$-module $\mathscr{E}=\mathbb{C}\oplus\mathbb{C}^2$ with module action 
		\[
		(x,u)\cdot(a,b)=(xa,ub), \qquad x\in\mathbb{C},\ u\in\mathbb{C}^2,\ (a,b)\in\mathscr{A},
		\] 
		and inner product 
		\[
		\langle (x,u),(y,v)\rangle = \bigl(\overline{x}y,\langle u,v\rangle_{\mathbb{C}^2}\bigr),
		\] 
		which is conjugate-linear in the first variable and linear in the second. This is a full Hilbert $\mathscr{A}$-module.
		
		We claim that $\mathscr{E}$ admits no $\sharp$-conjugation. Suppose, to the contrary, that there exists $C:\mathscr{E}\to\mathscr{E}$ such that $C^2=I_{\mathscr{E}}$ and $\langle C\xi,C\eta\rangle = \langle \xi,\eta\rangle^\sharp,\,\, \xi,\eta\in\mathscr{E}$.
		
		Put $$\mathscr{E}_1=\mathbb{C}\oplus\{0\} \quad \mbox{and} \quad \mathscr{E}_2=\{0\}\oplus\mathbb{C}^2.$$ For $\xi=(x,0)\in\mathscr{E}_1$, we have $\langle \xi,\xi\rangle=(|x|^2,0)$. Hence, $\langle C\xi,C\xi\rangle = \langle \xi,\xi\rangle^\sharp = (0,|x|^2)$.
		Writing $C\xi=(y,v)\in\mathbb{C}\oplus\mathbb{C}^2$, we get $\langle C\xi,C\xi\rangle=(|y|^2,\|v\|^2)$. Hence, $|y|^2=0$, so $y=0$, meaning $C(\mathscr{E}_1)\subseteq \mathscr{E}_2$.
		
		Similarly, for $\eta=(0,u)\in\mathscr{E}_2$, we have $\langle \eta,\eta\rangle=(0,\|u\|^2)$, so $\langle C\eta,C\eta\rangle = (\|u\|^2,0)$. Writing $C\eta=(z,w)$, we get $(|z|^2,\|w\|^2)=(\|u\|^2,0)$, hence $w=0$, so $C(\mathscr{E}_2)\subseteq \mathscr{E}_1$. Since $C^2=I$, we have $\mathscr{E}_2=C(C(\mathscr{E}_2))\subseteq C(\mathscr{E}_1)$. Thus, $C(\mathscr{E}_1)=\mathscr{E}_2$. 
		
		Therefore, the restriction $C|_{\mathscr{E}_1}:\mathscr{E}_1\to\mathscr{E}_2$ is a bijective conjugate-linear dimension-preserving map. But $\dim_\mathbb{C}\mathscr{E}_1=1$ whereas $\dim_\mathbb{C}\mathscr{E}_2=2$, impossible. Hence, no such $C$ exists.
	\end{example}
	
	\begin{example} \label{mosop}
		Let $\mathscr{A}$ be a unital $C^*$-algebra which is not $*$-isomorphic to its opposite algebra $\mathscr{A}^{\mathrm{op}}$. Such $C^*$-algebras are known to exist, see \cite{PV}. Consider the standard Hilbert $\mathscr{A}$-module $\mathscr{E}=\mathscr{A}$. Suppose that $\sharp$ were a $*$-conjugate-automorphism of $\mathscr{A}$. Then the map
		
		\[ \Phi:\mathscr{A}\to\mathscr{A}^{\mathrm{op}}, \qquad \Phi(a)=(a^\sharp)^* \] 
		entails that $\mathscr{A}\cong\mathscr{A}^{\mathrm{op}}$, a contradiction. A priori, there is no conjugation associated with any $*$-conjugate-automorphism.
	\end{example}
	
	The next result is the main result of this section. It characterizes the conjugations on a Hilbert $C^*$-module associated to a $\sharp$ by the so-called real forms. We use the structures introduced in Theorem \ref{thm:characterization_star_conjugation}. Complexification of real Hilbert $C^*$-modules is discussed in \cite[Proposition 3.1]{AK}.
	
	\begin{theorem} 
		\label{thm:real-form-characterization}
		Let $\mathscr{E}$ be a right Hilbert $\mathscr{A}$-module, and let
		$\sharp$ be a $*$-conjugate-automorphism of $\mathscr{A}$.
		Put
		\[
		\mathscr{A}_\mathbb{R}
		=
		\{a\in\mathscr{A}:a^\sharp=a\}.
		\]
		The following assertions are equivalent.
		
		\begin{enumerate}
			\item
			$\mathscr{E}$ admits a $\sharp$-conjugation $C$.
			
			\item
			There exists a real Hilbert $\mathscr{A}_\mathbb{R}$-module
			$\mathscr{F}$ such that the completed complexification 
			$
			\mathscr{F}\otimes_\mathbb{R}\mathbb{C}
			$
			of $\mathscr{F}$ is isometrically isomorphic to $\mathscr{E}$ as a complex Hilbert
			$\mathscr{A}$-module, where
			\[
			\mathscr{A}
			\cong
			\mathscr{A}_\mathbb{R}\otimes_\mathbb{R}\mathbb{C}.
			\]
			
			\item
			There exists a closed real linear subspace $\mathscr F \subseteq \mathscr E$ such that $\mathscr F$ is a right $\mathscr A_{\mathbb R}$-module, $\langle x, y \rangle_{\mathscr E} \in \mathscr A_{\mathbb R}$ for all $x, y \in \mathscr F$, $\mathscr F$ is complete in the inherited norm, and $\mathscr E = \mathscr F \oplus i\mathscr F$ as real vector spaces.
			
			Moreover, if $C$ is fixed, then the associated real form is uniquely determined and equals
			$$
			\mathscr E_{\mathbb R}
			=
			\operatorname{Fix}(C)
			=
			\{x\in\mathscr E:Cx=x\}.
			$$
		\end{enumerate}
	\end{theorem}
	
	\begin{proof}
		(1) $\Longrightarrow$ (2).
		Assume that $C$ is a conjugation on $\mathscr{E}$. Define
		\[
		\mathscr{E}_\mathbb{R}
		:=
		\{x\in\mathscr{E}:Cx=x\}.
		\]
		
		If $x\in\mathscr{E}_\mathbb{R}$ and $a\in\mathscr{A}_\mathbb{R}$, then $C(xa)=C(x)a^\sharp=xa$, and so $xa\in\mathscr{E}_\mathbb{R}$. In addition, for $x,y\in\mathscr{E}_\mathbb{R}$, we have $\langle x,y\rangle^\sharp
		=\langle Cx,Cy\rangle=\langle x,y\rangle$, and so $\langle x,y\rangle\in\mathscr{A}_\mathbb{R}$. Now, it is evident that $\mathscr{E}_\mathbb{R}$ is a real Hilbert $\mathscr{A}_\mathbb{R}$-module.

		Since $C$ is isometric, its fixed-point space is closed. Now, for every $x\in\mathscr{E}$, define $x_1:=\frac{x+Cx}{2}\in\mathscr{E}_\mathbb{R}$ and $x_2:=\frac{x-Cx}{2i}\in\mathscr{E}_\mathbb{R}$. Then, $x=x_1+ix_2$. Hence, clearly, $\mathscr{E}=\mathscr{E}_\mathbb{R} \oplus i\mathscr{E}_\mathbb{R}$. Let us define $\Psi:
		\mathscr{E}_\mathbb{R}\otimes_\mathbb{R}\mathbb{C} \to \mathscr{E}$ by $\Psi(x\otimes\lambda)=\lambda x$.
		The preceding decomposition shows that $\Psi$ is a bijective complex-linear map. 
		
		It is a right $\mathscr{A}$-module map, since under the canonical identification $\mathscr{A}_\mathbb{R}\otimes_\mathbb{R}\mathbb{C} \cong\mathscr{A}$, we have
		\[
		\Psi((x\otimes\lambda)(a\otimes\mu))
		=
		\lambda\mu\,xa
		=
		\Psi(x\otimes\lambda)\,(\mu a).
		\]
		On elementary tensors,
		\[
		\begin{aligned}
			\left\langle
			x\otimes\lambda,
			y\otimes\mu
			\right\rangle
			&=
			\overline{\lambda}\mu
			\langle x,y\rangle_{\mathscr{E}_\mathbb{R}},
		\end{aligned}
		\]
		and therefore
		\[
		\begin{aligned}
			\left\langle
			\Psi(x\otimes\lambda),
			\Psi(y\otimes\mu)
			\right\rangle_{\mathscr{E}}
			&=
			\langle\lambda x,\mu y\rangle_{\mathscr{E}}\\
			&=
			\overline{\lambda}\mu
			\langle x,y\rangle_{\mathscr{E}_\mathbb{R}}.
		\end{aligned}
		\]
		Thus, $\Psi$ preserves inner products and hence is an isometric Hilbert
		$\mathscr{A}$-module isomorphism. Under this identification, the map $C(x\otimes\lambda)=x\otimes\overline{\lambda}$ corresponds exactly to the original conjugation $C$.
		
		(2) $\Longrightarrow$ (3). Let $\mathscr{E}\cong\mathscr{F}\otimes_\mathbb{R}\mathbb{C}$ as complex Hilbert $\mathscr{A}$-modules. Identify $\mathscr{E}$ with this
		complexification. Then $\mathscr{F}\cap i\mathscr{F}=\{0\}$ and every elementary tensor can be written as
		$x\otimes\lambda=({\rm Re}\lambda)x+i({\rm Im}\lambda)x$. Hence, $\mathscr{E}=\mathscr{F}\oplus i\mathscr{F}$ as real vector spaces. 
		
		(3) $\Longrightarrow$ (1). Let $\mathscr{F}$ be as in (3). Since $\mathscr{E}=\mathscr{F}\oplus i\mathscr{F}$, every $z\in\mathscr{E}$ has a unique representation $z=x+iy$, where $x,y\in\mathscr{F}$. We now define $Cz:=C(x+iy):=x-iy$. The directness of the sum shows that $C$ is well-defined. Clearly, $C(\lambda z)=\overline{\lambda}\,Cz$ and $C^2=I_{\mathscr{E}}$. 
		
		Now, suppose that $x,y,x',y'\in\mathscr{F}$. Since the inner product is conjugate-linear
		in the first variable, 
		\[
		\begin{aligned}
			\langle x+iy,x'+iy'\rangle
			={}&
			\langle x,x'\rangle
			+i\langle x,y'\rangle
			-i\langle y,x'\rangle
			+\langle y,y'\rangle.
		\end{aligned}
		\]
		Due to $\mathscr{F}$ is a real Hilbert $\mathscr{A}_\mathbb{R}$-module, all four inner products appearing here belong to
		$\mathscr{A}_\mathbb{R}$. Hence, $\sharp$ fixes each of them and conjugates the scalar $i$. Therefore,
		\[
		\begin{aligned}
			\langle x+iy,x'+iy'\rangle^\sharp
			={}&
			\langle x,x'\rangle
			-i\langle x,y'\rangle
			+i\langle y,x'\rangle
			+\langle y,y'\rangle.
		\end{aligned}
		\]
		On the other hand,
		\[
		\begin{aligned}
			\langle C(x+iy),C(x'+iy')\rangle
			&=
			\langle x-iy,x'-iy'\rangle\\
			&=
			\langle x,x'\rangle
			-i\langle x,y'\rangle
			+i\langle y,x'\rangle
			+\langle y,y'\rangle.
		\end{aligned}
		\]
		Consequently,
		\[
		\langle Cz,Cw\rangle
		=
		\langle z,w\rangle^\sharp
		\qquad
		(z,w\in\mathscr{E}).
		\]
		Thus, $C$ is a $\sharp$-conjugation.
		
		Finally, if $C$ is given, its fixed-point set is necessarily ${\rm Fix}(C)=\{x\in\mathscr{E}:Cx=x\}$, so the real form associated with a fixed conjugation is unique.
	\end{proof}
	
	The next result characterizes invariant submodules.
	
	\begin{corollary}
		\label{thm:invariant-submodule}
		Let $\mathscr{E}$ be a right Hilbert $\mathscr{A}$-module with a $\sharp$-conjugation
		$C$, and let $\mathscr{F}\subseteq\mathscr{E}$ be a closed
		$C$-invariant submodule, that is, $C(\mathscr{F})\subseteq\mathscr{F}$. Then $C(\mathscr{F})=\mathscr{F}$, and the restriction
		\[
		C_{\mathscr{F}}:=C|_{\mathscr{F}}
		\]
		is a $\sharp$-conjugation on $\mathscr{F}$. Furthermore, $\mathscr{F}_\mathbb{R}:=\mathscr{F}\cap\mathscr{E}_\mathbb{R}=\{x\in\mathscr{F}:Cx=x\}$ is a real Hilbert $\mathscr{A}_\mathbb{R}$-module and $\mathscr{F}
		=\mathscr{F}_\mathbb{R}\oplus i\mathscr{F}_\mathbb{R}$.
		
		Consequently, $\mathscr{F}\cong\mathscr{F}_\mathbb{R}\otimes_\mathbb{R}\mathbb{C}$ as complex Hilbert $\mathscr{A}$-modules.
	\end{corollary}
	
	\begin{proof}
		Since $C^2=I_{\mathscr{E}}$ and $C(\mathscr{F})\subseteq\mathscr{F}$, for every $x\in\mathscr{F}$, we have $x=C^2x\in C(\mathscr{F})$. Thus, $C(\mathscr{F})=\mathscr{F}$. Therefore, $C|_{\mathscr{F}}$ is a conjugation. The assertion concerning $\mathscr{F}_\mathbb{R}$ follows by applying Theorem~\ref{thm:real-form-characterization} to the Hilbert
		$\mathscr{A}$-module $\mathscr{F}$.
	\end{proof}
	
	\begin{example} Let $\mathscr{A}=\mathbb{C}$ with $\lambda^\sharp=\overline{\lambda}$, and let $\mathscr{E}=\mathbb{C}^2$. Define $C(z_1,z_2)=(\overline{z_2},\overline{z_1})$. Then $C$ is a conjugation. Indeed, $C^2=I_{\mathbb{C}^2}$, and with the usual inner product, $\langle Cx,Cy\rangle=\overline{\langle x,y\rangle}=\langle x,y\rangle^\sharp$. 
		
		Its fixed-point space is $\mathscr{E}_\mathbb{R}=\{(z,\overline z):z\in\mathbb{C}\}$. As a real Hilbert space this is isometrically isomorphic to $\mathbb{R}^2$.
	\end{example}
	
	\begin{remark}
		The preceding example illustrates that a Hilbert $C^*$-module may admit
		different conjugations. For example, on $\mathbb{C}^2$ the standard
		coordinatewise conjugation and the twisted conjugation above have
		different fixed-point spaces. Thus, the real form is uniquely determined
		by a chosen conjugation, but it is not an intrinsic object determined
		solely by the complex Hilbert space or Hilbert $C^*$-module.
		
		Furthermore, the characterization theorem does not assert that every Hilbert
		$\mathscr{A}$-module admits a $\sharp$-conjugation. Such a conjugation exists
		precisely when the module possesses a compatible real form over
		$\mathscr{A}_\mathbb{R}$. In particular, the existence problem is a
		genuine structural question about the module and a chosen
		$*$-conjugate-automorphism $\sharp$.
	\end{remark}
	
	\section{Conjugations on Hilbert $W^*$-modules}\label{sec4}
	
	We start this section with a concept compatible with the Hilbert space setting. 
	
	For each $T\in\mathbb{L}(\mathscr{E})$, let us define the map $T^C:=CT^*C$. Then, the map
	\begin{equation}\label{eqsharp}
		\cdot^C:\mathbb{L}(\mathscr{E})\to\mathbb{L}(\mathscr{E}),\qquad T\mapsto T^C,
	\end{equation}
	gives a linear $*$-anti-automorphism on the $C^*$-algebra $\mathbb{L}(\mathscr{E})$, since
	\[(\lambda T)^C = C(\lambda T)^*C = C(\overline{\lambda}T^*)C = \lambda CT^*C = \lambda T^C.\]
	It is noteworthy that this setting includes the unital $C^*$-algebra case as a special instance, since when $\mathscr{A}$ is unital, the $C^*$-algebra $\mathbb{L}(\mathscr{A})$ can be identified with $\mathscr{A}$ (see \cite[p.~10]{LAN}). Thus, the map $\cdot^C$ can be linked to the theory of linear $*$-anti-automorphisms on $C^*$-algebras.
	
	\begin{definition}\label{moxesk}
		\begin{itemize}\item[]
			\item [(i)]An element $T\in\mathbb{L}(\mathscr{E})$ is called \emph{$C$-normal} if $(T^*T)^C= TT^*$, or equivalently, 
			$CT^*TC = TT^*$. It follows from the uniqueness of the square root in $\mathbb{L}(\mathscr{E})$ holds that $C|T|^2 C = |T^*|^2$ if and only if $C|T|C = |T^*|$.
			
			\item[(ii)] An element $T\in\mathbb{L}(\mathscr{E})$ is called \emph{$C$-symmetric} if $T=T^C$.
			It is called \emph{symmetric} if it is $C$-symmetric with respect to some linear $*$-anti-automorphism $\cdot^C$. Clearly, a $C$-symmetric operator is $C$-normal.
		\end{itemize}
	\end{definition}
	
	For $x,y\in\mathscr{E}$, let $\theta_{x,y}(z)=x\langle y,z\rangle,\,\, z\in\mathscr{E}$. Then, $\theta_{x,y}\in\mathbb{L}(\mathscr{E}) $ with $\theta_{x,y}^*=\theta_{y,x}$. Such operators are analogous to rank-one operators on a Hilbert space but they are not really rank-one in the sense of Banach space operator. In the following example, we treat such operators in the setting of Hilbert $C^*$-modules equipped with a $\sharp$-conjugation.
	
	\begin{example}
		Let $C$ be a $\sharp$-conjugation on a Hilbert $\mathscr{A}$-module $\mathscr{E}$. 
		For $x,y\in\mathscr{E}$, we have
		\begin{align*}
			\langle C\theta_{Cx,Cy}Cz,w\rangle&=\langle \theta_{Cx,Cy}Cz, Cw \rangle^\sharp=\langle Cx\langle Cy,Cz\rangle, Cw\rangle^\sharp\\
			&=\langle Cz,Cy\rangle^\sharp\langle Cx, Cw\rangle^\sharp=\langle z,y\rangle \langle x,w\rangle=\langle x\langle y,z\rangle,w\rangle\\
			&=\langle \theta_{x,y}(z),w\rangle.
		\end{align*}
		Hence, $C\theta_{Cx,Cy}C=\theta_{x,y}$. Thus,
		\begin{align}\label{far1}
			\theta_{x,y}^C=C\theta_{y,x}C=\theta_{Cy,Cx}.
		\end{align}
		If $Cx=y$, then $Cy=x$, and so $\theta_{x,y}=C\theta_{y,x}C=C\theta_{x,y}^*C=\theta_{x,y}^C$. Therefore, $\theta_{x,y}$ is $C$-symmetric.
	\end{example}

	Next, we show that any conjugation preserves the closed-range property and establish several related closure results.
	
	\begin{proposition}\label{prop:closed-range-C-adjoint}
		Let $\mathscr{E}$ be a Hilbert $C^*$-module equipped with a $\sharp$-conjugation $C$. Let $T\in\mathbb{L}(\mathscr{E})$. Then:
		\begin{itemize}
			\item[(i)] $\overline{\ran(T^C)} = C\!\left(\overline{\ran(T^*)}\right)$;
			\item[(ii)] $\ran(T^C)$ is closed if and only if $\ran(T^*)$ is closed, and this occurs if and only if $\ran(T)$ is closed;
			\item[(iii)] $\ker(T^C) = C(\ker(T^*)) = C(\overline{\ran(T)}^\perp)$.
		\end{itemize}
		In particular, if $\ran(T)$ is closed, then $\ker(T^C) = C(\ker(T^*)) = C(\ran(T)^\perp)$.
	\end{proposition}
	
	\begin{proof}
		(i) Since $C$ is a surjective isometry and hence a homeomorphism, for every subset $S \subseteq \mathscr{E}$ we have $C(\overline{S}) = \overline{C(S)}$. Consequently, $S$ is closed if and only if $C(S)$ is closed. Therefore,
		$$
		C\!\left(\overline{\ran(T^*)}\right) = \overline{C\!\left(\ran(T^*)\right)} = \overline{\left(\ran(CT^*)\right)}= \overline{\ran(CT^*C)} = \overline{\ran(T^C)}.
		$$
		
		(ii) Since $T^C = CT^*C$, we have $\ran(T^C) = C\!\left(\ran(T^*)\right)$. Thus, $\ran(T^C)$ is closed if and only if $C\!\left(\ran(T^*)\right)$ is closed. Since $C$ is a homeomorphism, the latter is closed if and only if $\ran(T^*)$ is closed. Moreover, by \cite[Theorem 3.2]{LAN}, $\ran(T)$ is closed if and only if $\ran(T^*)$ is closed.
		
		(iii) For the first equality, for any $x \in \mathscr{E}$,
		$$
		T^C x = 0 \iff C T^* C x = 0 \iff T^* C x = 0 \iff C x \in \ker(T^*) \iff x \in C(\ker(T^*)).
		$$
		Hence, $\ker(T^C) = C(\ker(T^*))$. The identity $\ker(T) = \overline{\ran(T^*)}^\perp$ is known; see the proof of \cite[Theorem 2.3.3]{Manu}. The rest is immediate.
	\end{proof}
	
	In establishing the following result we use the fact that a closed submodule of $\mathscr{E}$ is orthogonally complemeneted if and only if it is the range of a projection in $\mathbb{L}(\mathscr{E})$; see \cite[Corollary 2.3.4]{Manu}.
	
	\begin{lemma}\label{lem:C-projection}
		Let $\mathscr{E}$ be a Hilbert $C^*$-module equipped with a $\sharp$-conjugation $C$, and let $M$ be a submodule of $\mathscr{E}$. Then $C(M)$ is orthogonally complemented if and only if so is $M$. 
		
		In this case, $C P_M C = P_{C(M)}$ and so $C(M^\perp) = C(M)^\perp$.
	\end{lemma}
	
	\begin{proof}
		Let $M$ be orthogonally complemented. Hence, $P_M\in\mathbb{L}(\mathscr{E})$. Since $C$ is a conjugation, the operator $C P_M C$ is self-adjoint and satisfies $(C P_M C)^2 = C P_M C$. Hence, $C P_M C$ is a projection in $ \mathbb{L}(\mathscr{E})$ whose range is therefore orthogonally complemented. Moreover, the bijectivity of $C$ implies that
		$$
		\ran(C P_M C) = C(\ran(P_M)) = C(M).
		$$
		Therefore, $C(M)$ is an orthogonally complemented submodule of $\mathscr{E}$, and consequently $C P_M C = P_{C(M)}$. Since, $C$ is involutive, we conclude that if $C(M)$ is orthogonally complemented, then so is $M=C(C(M))$. The last equality follows from the decompositions $\mathscr{E}=M\oplus M^\perp$ and $\mathscr{E}=C(\mathscr{E})=C(M) \oplus C(M^\perp)$.
	\end{proof}
	
	In \cite[Lemma 2.5]{KO2}, Ko, Lee, and Lee extended the Douglas majorization theorem to the setting of Hilbert spaces equipped with a conjugation, under certain mild hypotheses. The next result uses a Hilbert $C^*$-module version of Douglas theorem \cite[Theorem 2.4]{FMX} to give a characterization of $C$-hyponormality, which is new even in the ordinary Hilbert space setting; see also \cite{ESK, XU1}.
	
	\begin{theorem}\label{thm:norm-characterization} 
		Let $\mathscr{E}$ be a Hilbert $C^*$-module over a unital $C^*$-algebra $\mathscr{A}$, let $C$ be a conjugation on $\mathscr{E}$, and let $T, S\in\mathbb{L}(\mathscr{E})$ and $\overline{\ran(T)}$ is orthogonally complemented. Then the following assertions are equivalent.
		\begin{enumerate}
			\item[(i)] $T^*T \geq CS^*S C$.
			\item[(ii)] $\|Tx\| \geq \|SC x\|$ for every $x\in\mathscr{E}$.
		\end{enumerate}
	\end{theorem}
	\begin{proof}
		
		(i) $\Longrightarrow$ (ii). Suppose that $T^*T \geq CS^*SC$. For any $x \in \mathscr{E}$, we have$$\|Tx\|^2 = \|\langle Tx, Tx \rangle\| = \|\langle T^*T x, x \rangle\| \geq \|\langle CS^*SC x, x \rangle\| = \|\langle SC x, SC x \rangle\| = \|SC x\|^2.$$Hence, $\|Tx\| \geq \|SC x\|$. (In this proof, we do not use the assumption that $\overline{\ran(T)}$ is orthogonally complemented.)
		
		(ii) $\Longrightarrow$ (i). Suppose that $\|Tx\| \geq \|SC x\|$ for all $x \in \mathscr{E}$. Since $C$ is a conjugation, for any $x \in \mathscr{E}$ we have $\|Tx\| \geq \|SCx\| = \|CSCx\|$. Utilizing \cite[Theorem 2.4]{FMX} gives $$T^*T \geq (CS^*C)(CSC) = CS^*SC.$$
	\end{proof}
	
	\begin{corollary}\label{thm:norm-characterization1}
		Let $\mathscr{E}$ be a Hilbert $C^*$-module over a unital $C^*$-algebra $\mathscr{A}$ admitting a conjugation $C$, and let $T\in\mathbb{L}(\mathscr{E})$ such that $\overline{\ran(T)}$ is orthogonally complemented. Then the following assertions are equivalent.
		\begin{enumerate}\item[(i)] $T$ is $C$-hyponormal, that is, $T^*T \geq C T T^* C$.\item[(ii)] $\|Tx\| \geq \|T^*C x\|$ for every $x\in\mathscr{E}$.
		\end{enumerate}
	\end{corollary}
	
	Now, we present a characterization of $C$-symmetric forms in certain Hilbert $C^*$-modules, which is a counterpart of \cite[Lemma 2.27]{G4}.
	
	For a conjugation $C$ on a Hilbert $\mathscr{A}$-module $\mathscr{E}$, define
	\[
	[x,y]:=\langle Cx,y\rangle,\qquad x,y\in\mathscr{E}.
	\]
	Then $[\cdot,\cdot]$ is bilinear and satisfies $[x,ya]=[x,y]a$ and $[xa,y]=a^{\sharp *}[x,y]$ for all $a\in\mathscr{A}$. Moreover, $[x,y]=[y,x]^{*\sharp}$, since
	\[
	[y,x]^{*\sharp}=\langle Cy,x\rangle^{*\sharp}=\langle x,Cy\rangle^\sharp=\langle Cx,y\rangle=[x,y].
	\]
	Thus, in general, $[\cdot,\cdot]$ is not an ordinary symmetric bilinear form unless
	\[
	a^{*\sharp}=a.
	\]

	Let us establish a module analogue of the correspondence between symmetric bilinear forms and $C$-symmetric operators in the setting of Hilbert spaces.

	\begin{theorem}\label{th:main_corrected}
		Let $\mathscr{M}$ be a $W^*$-algebra, let $\mathscr{E}$ be a self-dual Hilbert $\mathscr{M}$-module equipped with a conjugation $C$ associated with a $*$-conjugate-automorphism $\sharp$ of $\mathscr{M}$. 
		Suppose that $\Gamma:\mathscr{E}\times\mathscr{E}\to\mathscr{M}$ is a bounded form that is bilinear and $\mathscr{M}$-linear in the second variable ($\Gamma(x,ya)=\Gamma(x,y)a$), and symmetric ($\Gamma(x,y)^\sharp=\Gamma(y,x)^*$). Then there exists a unique $C$-symmetric operator $T\in\mathbb{L}(\mathscr{E})$ satisfying
		\begin{align}\label{bil}
			\Gamma(x,y) =[Tx,y],\qquad x,y\in\mathscr{E}.
		\end{align}
		Conversely, every $C$-symmetric operator $T\in\mathbb{L}(\mathscr{E})$ determines a bounded symmetric bilinear form which is $\mathscr{M}$-linear in both variables.
	\end{theorem}
	
	\begin{proof}
		For each fixed $x\in\mathscr{E}$, we set
		\[
		\Phi_x:\mathscr{E}\to \mathscr{M},\qquad
		\Phi_x(y):=\Gamma(Cx,y).
		\]
		Since $\Gamma$ is bounded, linear, and $\mathscr{M}$-linear in the second variable,
		$\Phi_x$ is a bounded $\mathscr{M}$-linear functional on $\mathscr{E}$.
		By self-duality of $\mathscr{E}$, there exists a unique vector
		$Sx\in\mathscr{E}$ with $\|Sx\|=\|\Phi_x\|$ such that
		$$\Gamma(Cx,y)=\langle Sx,y\rangle \qquad y\in\mathscr{E}.$$ 
		Therefore,
		\[\Gamma(x,y)=\langle SCx,y\rangle,\qquad x,y \in \mathscr{E}.\]
		
		Thus, a map $S:\mathscr{E}\to\mathscr{E}$ can be defined by $x\mapsto Sx$. The map $S$ is linear since 
		$$\langle S(\lambda x),y\rangle=\Gamma(C(\lambda x),y)=\Gamma(\overline{\lambda}Cx,y)=\overline{\lambda}\Gamma(Cx,y)=\overline{\lambda}\langle Sx,y\rangle =\langle \lambda Sx,y\rangle$$
		for $\lambda\in\mathbb{C}$. It is evidently bounded with $\|S\|=\|\Gamma\|$. Since $\Gamma$ is symmetric, we have 
		\begin{align*}
			\langle S(xa),y\rangle^\sharp&=\Gamma(C(xa),y)^\sharp=\Gamma(y, C(xa))^*=\Gamma(y,Cxa^\sharp)^*=(\Gamma(y,Cx)a^\sharp)^*\\
			&=(a^\sharp)^*\Gamma(y,Cx)^*=(a^\sharp)^*\Gamma(Cx,y)^\sharp =(a^*)^\sharp\langle Sx,y\rangle^\sharp\\
			&=(a^*\langle Sx,y\rangle)^\sharp=\langle Sx a,y\rangle^{\sharp}
		\end{align*} 
		whence $ S(xa)=Sx a$.

		Since every bounded $\mathscr{M}$-linear map on a self-dual Hilbert
		$\mathscr{M}$-module is adjointable \cite[Corollary 3.5]{PAS}, we conclude that
		$S\in\mathbb{L}(\mathscr{E})$. Set $T=S^*$. In addition, 
		\begin{align*}
			\langle T^*Cx,y\rangle^\sharp&=\Gamma(x,y)^\sharp=\Gamma(y,x)^*=\langle SCy,x\rangle^*=\langle S^*x,Cy\rangle=\langle CS^*x,y\rangle^\sharp\\
			& =\langle CTx,y\rangle^\sharp=[Tx,y]^\sharp.
		\end{align*} 
		Hence, $T^*C=CT$, and so $T$ is $C$-symmetric. In addition, $\Gamma(x,y)=[Tx,y]$ as required.
		
		The uniqueness follows immediately from the nondegeneracy of the inner product. In fact, let $T, T'\in\mathbb{L}(\mathscr{E})$ satisfies $\langle CTx,y\rangle=[Tx,y]=\Gamma(x,y)=[T'x,y]=\langle CT'x,y\rangle$ for all $x,y\in \mathscr{E}$. Then $\langle C(T-T')x,y\rangle=0$ for every $x,y$. Hence, $C(T-T')x=0$ for all $x$. It follows from injectivity of $C$ that $T=T'$.
		
		Conversely, let $T\in\mathbb{L}(\mathscr{E})$ satisfy $T=CT^*C$, and let us define $\Gamma_T(x,y):=[T^*x,y]=\langle CT^*x, y\rangle$. Then,
		\[
		\|\Gamma_T(x,y)\| =\|\langle CT^*x,y\rangle\| \leq \|T^*\|\,\|x\|\,\|y\|,
		\]
		so $\Gamma_T$ is bounded. The bilinearity and $\mathscr{M}$-linearity of $\Gamma$ in the second variable can be verified. Let us prove only the $\mathscr{M}$-linearity on the second variable: 
		$$\Gamma_T(x,ya)=\langle CT^*(x),ya\rangle = \langle CT^*x, y\rangle a=\Gamma_T(x,y)a,$$
		Since $T^*\in\mathbb{L}(\mathscr{E})$ we have 
		\begin{align*}
			\Gamma_T(x,y)^\sharp&=\langle CT^*x,y\rangle^\sharp=\langle T^*x,Cy\rangle=\langle Cy,T^*x\rangle^*\\
			&=\langle TCy,x\rangle^*=\langle CT^*y,x\rangle^*=\Gamma_T(y,x)^*
		\end{align*}
		from which we conclude that $\Gamma_T$ is symmetric.
	\end{proof}
	
	Now, we explore the $\sharp$-conjugate dual.
	
	\begin{definition}
		Set
		\[
		\mathscr{E}^{\prime}_{c,\sharp}
		=
		\left\{
		f:\mathscr{E}\to\mathscr{A}:
		\begin{array}{l}
			f\text{ is bounded and conjugate-linear},\\
			f(xa)=f(x)a^\sharp
			\quad (x\in\mathscr{E},\ a\in\mathscr{A})
		\end{array}
		\right\}.
		\]
		We call $\mathscr{E}^{\prime}_{c,\sharp}$ the
		\emph{$\sharp$-conjugate dual} of $\mathscr{E}$.
	\end{definition}
	
	We equip $\mathscr{E}^{\prime}_{c,\sharp}$ with the right
	$\mathscr{A}$-module structure
	\[
	(fa)(x)=a^*f(x),
	\qquad
	f\in\mathscr{E}^{\prime}_{c,\sharp},
	\quad
	a\in\mathscr{A},
	\quad
	x\in\mathscr{E}.
	\]
	This is well-defined, since
	\begin{align*}
		(fa)(xb)
		&=
		a^*f(xb)=a^*f(x)b^\sharp=(fa)(x)b^\sharp.
	\end{align*}
	Set
	\[
	\mathscr{E}':=
	\left\{
	g:\mathscr{E}\to\mathscr{A}:
	\begin{array}{l}
		g\text{ is bounded and linear},\\
		g(xa)=g(x)a\quad (x\in\mathscr{E},\ a\in\mathscr{A})
	\end{array}
	\right\}
	\]
	be the ordinary $\mathscr{A}$-module dual, equipped with the right
	$\mathscr{A}$-module structure
	\[
	(ga)(x)=a^*g(x).
	\]
	
	\begin{theorem}\label{thm:conjugate_dual_isomorphism}
		Let $\mathscr{E}$ be a Hilbert $\mathscr{A}$-module equipped with a
		conjugation $C$ associated with $\sharp$. Then
		\[
		\Phi:\mathscr{E}^{\prime}_{c,\sharp}\to\mathscr{E}',
		\qquad\Phi(f)=f\circ C,
		\]
		is a bijective right $\mathscr{A}$-module map. Its inverse is $\Phi^{-1}(g)=g\circ C$. Consequently, 
		\[
		\mathscr{E}^{\prime}_{c,\sharp}
		\cong
		\mathscr{E}'
		\]
		as right $\mathscr{A}$-modules. By considering the operator norm on $\mathscr{E}'$ and $\mathscr{E}^{\prime}_{c,\sharp}$, $\Phi$ is isometric.
	\end{theorem}
	
	\begin{proof}
		Let $f\in\mathscr{E}^{\prime}_{c,\sharp}$. Since both $f$ and $C$
		are conjugate-linear, $f\circ C$ is linear. Moreover,
		$$(f\circ C)(xa)=f(C(xa))=f(Cx\,a^\sharp)=f(Cx)(a^\sharp)^\sharp=f(Cx)a=(f\circ C)(x)a.$$
		Hence, $f\circ C\in\mathscr{E}'$.
		
		Conversely, if $g\in\mathscr{E}'$, then $g\circ C$ is
		conjugate-linear and
		$$(g\circ C)(xa)=g(C(xa))=g(Cx\,a^\sharp)=g(Cx)a^\sharp=(g\circ C)(x)a^\sharp.$$
		Thus
		\[
		g\circ C\in\mathscr{E}^{\prime}_{c,\sharp}.
		\]
		Since $C^2=I$, $\Phi$ is bijective and $\Phi^{-1}(g)=g\circ C$. Moreover, 
		$$\Phi(fa)(x)=(fa)(Cx)=a^*f(Cx)=a^*\Phi(f)(x)=(\Phi(f)a)(x).$$
		Therefore, $\Phi(fa)=\Phi(f)a$,so $\Phi$ is a right $\mathscr{A}$-module isomorphism. 
		
		In addition, $\|\Phi(f)\|=\|f \circ C\| \leq \|f\|\,\|C\|\leq \|f\|$ and $\|f\|=\|f\circ (C \circ C)\|=\|\Phi(f) \circ C\|\leq \|\Phi(f)\|\,\|C\|\leq \|\Phi(f)\|$, since $C$ is involutive and isometric. Thus, $\|\Phi(f)\|=\|f\|$.
	\end{proof}

	For each $x\in\mathscr{E}$, let us define $\widehat{x}:\mathscr{E}\to\mathscr{A}$ by 
	$$\widehat{x}(y)=\langle x,Cy\rangle.$$
	Then, $\widehat{x}\in \mathscr{E}^{\prime}_{c,\sharp}$ and 
	$$\Phi(\widehat{x})(y)=(\widehat{x}\circ C)(y)=\langle x,C(Cy)\rangle=\langle x,y\rangle.$$
	Consequently, under the isomorphism $\Phi$, $x \hookrightarrow \Phi(\widehat{x})$ gives an algebraic embedding of $\mathscr{E}$ into $\mathscr{E}^{\prime}$ corresponding to the usual canonical embedding of $\mathscr{E}$ into its ordinary dual.

	We now impose the additional hypothesis that $\mathscr{A}$ is a $W^*$-algebra. Theorem~3.2 of \cite{PAS} gives an $\mathscr{A}$-valued inner product on the ordinary dual $\mathscr{E}'$ which extends the original inner product and makes $\mathscr{E}'$ a self-dual Hilbert $\mathscr{A}$-module. We denote this inner product by
	$\langle\cdot,\cdot\rangle_P$. We can transport this inner product to $\mathscr{E}^{\prime}_{c,\sharp}$ through the isomorphism $\Phi$.
	
	\begin{theorem}\label{thm:conjugate_dual_inner_product}
		Let $\mathscr{A}$ be a $W^*$-algebra equipped with a $*$-conjugate-automorphism $\sharp$, and let $\mathscr{E}$ be a right Hilbert $\mathscr{A}$-module equipped with a $\sharp$-conjugation $C$. Then the formula
		\[ \langle f,g\rangle_{c,\sharp}=\langle f\circ C,g\circ C\rangle_P \]
		defines an $\mathscr{A}$-valued inner product on the right $\mathscr{A}$-module $\mathscr{E}^{\prime}_{c,\sharp}$, making it into a self-dual Hilbert $\mathscr{A}$-module.
	\end{theorem}
	
	\begin{proof}
		In light of Theorem~\ref{thm:conjugate_dual_isomorphism}, $\Phi(f)=f\circ C$ is a right $\mathscr{A}$-module isomorphism from $\mathscr{E}^{\prime}_{c,\sharp}$ onto $\mathscr{E}'$. Hence, we can define 
		\[\langle f,g\rangle_{c,\sharp}=\langle\Phi(f),\Phi(g)\rangle_P.\]
		Since $\langle\cdot,\cdot\rangle_P$ is an $\mathscr{A}$-valued inner product, it is immediate that 
		$\mathscr{E}^{\prime}_{c,\sharp}$ is an inner product $\mathscr{A}$-module. Since $\Phi$ is isometric, the norm induced by the inner product on $\mathscr{E}^{\prime}_{c,\sharp}$ is complete.
		
		It remains to prove self-duality. Let $F:\mathscr{E}^{\prime}_{c,\sharp}\to\mathscr{A}$
		be a bounded $\mathscr{A}$-linear map. We define $G:\mathscr{E}'\to\mathscr{A}$ by
		\[G(h)=F(\Phi^{-1}(h)).\]
		Since $\Phi^{-1}$ is a right $\mathscr{A}$-module isometric isomorphism, $G$ is a bounded $\mathscr{A}$-linear map. By the self-duality of $\mathscr{E}'$, there exists $g\in\mathscr{E}'$ such that
		\[G(h)=\langle g,h\rangle_P\qquad(h\in\mathscr{E}').\]
		Set $f:=\Phi^{-1}(g)\in\mathscr{E}^{\prime}_{c,\sharp}$.Then, for every $u\in\mathscr{E}^{\prime}_{c,\sharp}$,
		$$F(u)=G(\Phi(u))=\langle g,\Phi(u)\rangle_P=\langle\Phi(f),\Phi(u)\rangle_P=\langle f,u\rangle_{c,\sharp}.$$
		Hence, every bounded $\mathscr{A}$-linear functional on $\mathscr{E}^{\prime}_{c,\sharp}$ is represented by an element of $\mathscr{E}^{\prime}_{c,\sharp}$. Therefore, $\mathscr{E}^{\prime}_{c,\sharp}$ is self-dual.
	\end{proof}

	\section{Polar decomposition of operators} \label{sec5}
	
	Garcia and Putinar established an interesting result \cite[Theorem 2]{G2}. It states that if $T = U|T|$ is the polar decomposition of a $C$-symmetric operator $T$, then $T = CJ|T|$, where $J$ is a conjugate-partial isometry supported on $\overline{\ran(|T|)}$ and commutes with $|T| = (T^*T)^{1/2}$. In particular, the partial isometry $U$ is $C$-symmetric and factors as $U = CJ$.
	
	Attempting to obtain a similar result in the setting of Hilbert $C^*$-modules with a $\sharp$-conjugation encounters several pathological issues. These arise from the geometric properties of Hilbert $C^*$-modules and the operators acting on them. In contrast with the Hilbert space setting, the polar decomposition does not generally hold in $\mathbb{L}(\mathscr{E})$, the equality $\ker(T)^\perp=\overline{\ran(T^*)}$ generally does not hold in Hilbert $C^*$-modules, a bounded $\mathscr{A}$-linear operator on a Hilbert module need not be adjointable, and a closed submodule may fail to be complemented. 
	
	To get an analogous result to \cite[Theorem 2]{G2}, we need some auxiliary results. 
	
	We recall that an element $U \in \mathbb{L}(\mathscr{E})$ is called a \emph{partial isometry} if $\ran(U)$ is orthogonally complemented and $U$ is isometric on $\ker(U)^\perp$. It is known \cite[p. 30]{LAN} that $U$ is a partial isometry if and only if either $UU^*$ is a projection, or equivalently, $UU^*$ is a projection, and this occurs if and only if $UU^*U=U$.
	
	\begin{definition}
		An operator $T$ in $\mathbb{L}(\mathscr{E})$ or $\mathbb{L}_c(\mathscr{E})$ is \emph{semiregular} if $\overline{\ran(T)}$ and $\overline{\ran(T^*)}$ are orthogonally complemented.
	\end{definition} 
	
	From now on, for an orthogonally complemented submodule $M$, we denote by $P_M$ the projection $P\in \mathbb{L}(\mathscr{E})$ with $\ran(P)=M$. We require the following result; see \cite{FS} for the polar decomposition of unbounded operators.

	\begin{lemma}\cite[Lemma 3.6 and Theorem 3.8]{LIULu}\label{lemXu}
		Let $T\in\mathbb{L}(\mathscr{E})$ be semiregular. Then there exists a partial isometry $U\in\mathbb{L}(\mathscr{E})$ such that
		$$
		T=U|T| \quad\text{and}\quad U^*U=P_{\overline{\ran{(T^* )}}}.
		$$
		Moreover, such a decomposition of $T$ as $T=UQ$ with $Q$ positive and $U$ satisfying the above condition is unique. Furthermore,
		\begin{itemize}
			\item[(i)] all equations $\ker{(U)} = \ker{(T)}$, $\ker{(U^* )} = \ker{(T^* )}$, $\ran{U}=\overline{\ran{(T)}}=\overline{\ran(|T^*|)}$, and $\ran{(U^* )} = \overline{\ran{(T^* )}}=\overline{\ran(|T|)}$ are satisfied;
			
			\item[(ii)] the following equations are also valid:
			$$
			T^*=U^*|T^*| \quad\text{and}\quad UU^*=P_{\overline{\ran{(T)}}},
			$$
			$$
			|T^*|=U|T|U^* \quad\text{and}\quad U|T|=|T^*|U.
			$$
		\end{itemize}
		
	\end{lemma}

	We require the notion of conjugate-partial isometry for our investigation. To explore its properties, we require a lemma.
	
	\begin{lemma}\label{theorthogonal}
		Let $\mathscr{E}$ be a Hilbert $\mathscr{A}$-module endowed with a $\sharp$-conjugation $C$. Suppose that $S\in \mathbb{L}_c(\mathscr{E})$ has closed range. Then $\ran(S)$, $\ker(S)$ , $\ran(S^*)$, and $\ker(S^*)$ are orthogonally complemented submodules.
	\end{lemma}
	\begin{proof}
		Since $SC\in \mathbb{L}(\mathscr{E})$ and $\ran(S)=\ran(SC)$, so $\ran(SC)$ is closed. By \cite[Theorem 3.2]{LAN}, $\ran(CS^*)$ is closed. Hence, $C(\ran(S^*))$ is closed, since $C^2=I$ and $C$ is an isometry we conclude that $\ran(S^*)$ is closed. 
		
		Now by the same argument as in the proof of \cite[Theorem 3.2]{LAN}, one can show that $\ran(S)$ and $\ran(S^*)$ are orthogonally complemented submodules of $\mathscr{E}$ as well as $\ran(S^*)\oplus\ker(S)=\mathscr{E}=\ran(S)\oplus\ker(S^*)$, and so $\ker(S)$ and $\ker(S^*)$ are also complemented submodules.
	\end{proof}

	\begin{definition}\label{defpartial}
		Let $\mathscr{E}$ be a Hilbert $C^*$-module over $\mathscr{A}$. An operator $J\in\mathbb{L}_c(\mathscr{E})$ is called a \emph{conjugate-partial isometry} if
		\[
		\langle Jx,Jy\rangle=\langle x,y\rangle^\sharp 
		\]
		for all $x,y\in\ker(J)^\perp$.
		
		The submodules $\ker(J)^\perp$ and $\ran(J)$ are called the initial and final submodules of $J$, respectively. A conjugate-partial isometry $J$ is called a \emph{partial conjugation} if $J^2x=x$ for all $x\in\ker(J)^\perp$.
	\end{definition}

	\begin{lemma}\cite[Lemma 2.3]{XU1}\label{lemranT^alpha}
		Let $A\in\mathbb{L}(\mathscr{E})$ be positive. Then
		\[
		\overline{\ran(A^\alpha)}=\overline{\ran(A)}\qquad(0 <\alpha \leq 1).
		\]
		In particular, $\overline{\ran(|T|)}=\overline{\ran(T^*T)}$.
	\end{lemma}
	
	\begin{theorem}\label{lemmall}
		Let $\mathscr{E}$ be a Hilbert $C^*$-module equipped with a $\sharp$-conjugation $C$, and let $T\in\mathbb{L}_c(\mathscr{E})$ be semiregular. Then $T$ admits a polar decomposition
		$T=V|T|$, where $V\in\mathbb{L}_c(\mathscr{E})$ is a conjugate partial isometry with the initial submodule $\overline{\ran(T^*)}$, in particular, $V^*T=|T|$. Furthermore, $$T=CJ|T|,$$ where $J$ is a partial isometry with $J^*J=P$, where $P := V^*V$ is the initial projection of $V$. In addition, if $V$ is $C$-symmetric, then $J^2 = P$ and $JP= PJ = J$.
	\end{theorem}
	
	\begin{proof}
		Consider the linear operator $S:=CT$. Since $C$ and $T$ are both conjugate-linear, $S$ is linear. Moreover, $S^*=T^*C$. Due to $C$ is a conjugation, in light of Lemma \ref{lem:C-projection}, $\overline{\ran(S^*)}=\overline{\ran(T^*C)}=\overline{\ran(T^*)}$ and $\overline{\ran(S)}=\overline{\ran(CT)}=\overline{C(\ran(T))}=C(\overline{\ran(T)})$ are orthogonally complemented. So $S\in\mathbb{L}(\mathscr{E})$ is semiregular.
		
		By Lemma \ref{lemXu}, $S$ has a polar decomposition $S=U|S|$, where $U\in\mathbb{L}(\mathscr{E})$ is a partial isometry with $U^*U=P_{\overline{\ran(|S|)}}$. Furthermore,
		$$
		|S|
		=
		(S^*S)^{1/2}
		=
		(T^*CC T)^{1/2}
		=
		(T^*T)^{1/2}
		=
		|T|.
		$$
		Therefore, $CT=U|T|$, and so $T=CU|T|$. Define $$V:=CU.$$
		Since $C$ is conjugate-linear and $U$ is linear, $V$ is conjugate-linear and $V^*V=U^*CCU=U^*U=P_{\overline{\ran(|T|)}}$, in particular, $V^*V|T|=|T|$. Thus, $T=V|T|$ and $V^*T=|T|$.
		
		It remains to show that $V$ is a conjugate-partial isometry. 
		
		Since $TC\in \mathbb{L}(\mathscr{E})$, \cite[Proposition 3.7]{LAN} yields $\overline{\ran((CT)^*)}=\overline{\ran((CT)^*CT)}=\overline{\ran(T^*T)}$. On the other hand, since $C$ is surjective $\overline{\ran((CT)^*)}=\overline{\ran(T^*C)}=\overline{\ran(T^*)}$. Therefore, in light of Lemma \ref{lemranT^alpha},
		$$\overline{\ran(T^*)}=\overline{\ran(T^*T)}=\overline{\ran(|T|)}.$$

		Let $x,y$ belong to the initial submodule of $U$, namely $\overline{\ran(|T|)}=\overline{\ran (T^*)}$. Then 
		$$
		\langle Vx,Vy\rangle
		=
		\langle CUx,CUy\rangle
		=
		\langle Ux,Uy\rangle^\sharp
		=
		\langle x,y\rangle^\sharp,
		$$
		because $U$ is an isometry on its initial submodule. Thus, $V$ is a conjugate-partial isometry.
		
		Now, if we set $J:=U$, then $J$ is a partial isometry with required properties and $T=CJ|T|$.
		
		Now assume that $V$ is $C$-symmetric. Then $V C = C V^*$ and $C V = V^* C$. Therefore,
		$$
		J = C V = V^* C = (C V)^* = U^* = J^*.
		$$
		
		Let $P:=P_{\overline{\ran(|T|)}}=P_{\overline{\ran(T^*)}}$. Then
		$$
		J^2 = (C V)(C V) = (V^* C)(C V) = V^* V = P.
		$$
		We have $JP= PJ = J$, since $J=CV=C(VV^*V)=(CV)(V^* V) =JP = V^* C V^* V = V^*V CV = PJ$. 
	\end{proof}

	Next, we present another main result of this paper.

	\begin{theorem}\label{thm:decomposition}
		Let $\mathscr{A}$ be a $C^*$-algebra, let $\mathscr{E}$ be a Hilbert
		$\mathscr{A}$-module equipped with a $\sharp$-conjugation $C$, and let
		$T\in\mathbb{L}(\mathscr{E})$ be semiregular. Suppose that
		$T=U|T|$ is the polar decomposition of $T$, where $U$ is a partial isometry
		with initial submodule $\overline{\ran(T^*)}$.
		Then there exists a conjugate-partial isometry $J$ such that
		\begin{enumerate}
			\item $J^*J=P$, where $P:=U^*U$ is the initial projection of $U$;
			\item $U=CJ$;
			\item $T=CJ|T|$.
		\end{enumerate}
		Moreover, if $U$ is $C$-symmetric, then $J^2=P$ and $JP=PJ=J$.
	\end{theorem}
	
	\begin{proof}
		Since $C$ is a conjugation, the operator $CTC$ belongs to
		$\mathbb{L}(\mathscr{E})$. By the uniqueness of the positive square root,
		$|CTC| = ((CTC)^*(CTC))^{1/2} = C(T^*T)^{1/2}C = C|T|C$.
		
		The semiregularity of $T$ is preserved under conjugation by $C$.
		Hence, $CTC$ is semiregular. Its polar decomposition is therefore
		\begin{equation}\label{eq:CTC-polar}
			CTC=V|CTC|=VC|T|C,
		\end{equation}
		where $V\in\mathbb{L}(\mathscr{E})$ is a partial isometry whose
		initial submodule is
		$\overline{\ran((CTC)^*)}
		=\overline{\ran(CT^*C)}
		=C(\overline{\ran(T^*)})$.
		
		It follows from \eqref{eq:CTC-polar} that $T=CVC|T|$. We claim that $CVC$ is a partial isometry whose 
		initial submodule is $\overline{\ran(T^*)}$. Indeed,
		$(CVC)^*(CVC)=CV^*VC$. Since $V^*V$ is the projection onto
		$C(\overline{\ran(T^*)})$, the operator $CV^*VC$
		is the projection onto $\overline{\ran(T^*)}$.
		Thus, $CVC$ is a partial isometry with initial submodule
		$\overline{\ran(T^*)}$.
		
		Therefore, $T=(CVC)|T|$ is a polar decomposition of $T$ with the 
		same prescribed initial submodule as $T=U|T|$. By the uniqueness of 
		the polar decomposition, $CVC=U$.
		
		Define $J:=VC$. Then $J$ is conjugate-linear. Moreover, since $V$ 
		is a partial isometry, $J$ is a conjugate-partial isometry. From 
		$J=VC$, we derive that $CJ=CVC=U$, and hence $T=U|T|=CJ|T|$. This proves 
		assertions (2) and (3).
		
		Let $P:=U^*U$. Since $U=CVC$, we have
		$P = U^*U = (CVC)^*(CVC) = CV^*VC$. On the other hand, $J^*J = CV^*VC = P$. This proves assertion (1).
		
		Now suppose that $U$ is $C$-symmetric. Then $U=CJ$ yields $U=CU^*C=CJ^*CC=CJ^*$. Hence, $CJ=U=CJ^*$, and therefore $J=J^*$. Combining this with $J^*J=P$ gives $J^2=P$.
		
		Finally, since $J$ is a conjugate-partial isometry, it satisfies $JJ^*J=J$. As $J=J^*$, this becomes $J^3=J$. Using $P=J^2$, we consequently obtain $JP=JJ^2=J^3=J$ and $PJ=J^3=J$. Thus, $JP=PJ=J$.
	\end{proof}

	\begin{proposition}
		Let $\mathscr{E}$ be a Hilbert $C^*$-module equipped with a
		conjugation $C$, and let $T\in\mathbb{L}(\mathscr{E})$ be
		semiregular. Suppose that $T=U|T|$ is the polar decomposition of $T$. 
		If $T$ is $C$-symmetric, then so is $U$.
	\end{proposition}
	
	\begin{proof}
		Since $T$ is $C$-symmetric, we have $T^*T = CTC\,CT^*C = C(TT^*)C$. Taking the unique positive square roots, 
		we obtain $|T| = C|T^*|C$. On the other hand, the polar decomposition of $T$ gives $T^*=U^*|T^*|$. 
		Hence,$T = CT^*C = CU^*C\,C|T^*|C = CU^*C\,|T|$. 
		
		We now verify that $CU^*C$ has the same initial submodule as $U$. 
		Since $U$ is the partial isometry in the polar decomposition of $T$, 
		its initial projection is the projection onto 
		$\overline{\ran(T^*)} = \overline{\ran(|T|)}$. 
		The final submodule of $U$ is $\overline{\ran(T)}$. 
		Since $T=CT^*C$, we have 
		$C\bigl(\overline{\ran(T)}\bigr) 
		= \overline{\ran(T^*)}$. The initial submodule of 
		$CU^*C$ is 
		$C\bigl(\overline{\ran(U)}\bigr) 
		= C\bigl(\overline{\ran(T)}\bigr) 
		= \overline{\ran(T^*)}$. Hence, $CU^*C$ is a partial 
		isometry with the same initial submodule as $U$. By the uniqueness of the 
		polar decomposition, $U=CU^*C$. Hence, $U$ is $C$-symmetric.
	\end{proof}
	
	It can be easily verified that if an operator $U=CJ$ is the product of two $\sharp$-conjugations $C$ and $J$, then $U$ is unitary and is both $C$-symmetric and $J$-symmetric; see \cite[Lemma 3.2]{G4}. Therefore, the condition of being $C$-symmetric in the following corollary is not a restrictive assumption. 
	
	\begin{corollary}\label{2con}
		Let $\mathscr{A}$ be a $C^*$-algebra, let $\mathscr{E}$ be a Hilbert
		$\mathscr{A}$-module equipped with a $\sharp$-conjugation $C$, and let
		$U\in\mathbb{L}(\mathscr{E})$ be a $C$-symmetric unitary. Then $U$ is
		the product of two $\sharp$-conjugations.
	\end{corollary}
	
	\begin{proof}
		Since $U$ is unitary, $U=U|U|$ is its polar decomposition. In particular, $U$ is 
		semiregular and $\overline{\ran(U^*)}
		=\overline{\ran(U)}=\mathscr{E}$. Applying Theorem~\ref{thm:decomposition} to $U$, we obtain a 
		conjugate-partial isometry $J$ such that $U=CJ$ and 
		$J^*J=P$, where $P$ is the projection onto 
		$\overline{\ran(U^*)}$. Therefore, $P=I_{\mathscr{E}}$. Thus, $J^*J=I_{\mathscr{E}}$.
		
		Moreover, $U$ is $C$-symmetric, so the last assertion of 
		Theorem~\ref{thm:decomposition} gives $J^2=P=I_{\mathscr{E}}$. 
		Therefore, $J$ is a conjugation on $\mathscr{E}$.
	\end{proof}
	
	The following example shows that, in contrast to the Hilbert space setting, not every unitary on a Hilbert $C^*$-module is a product of two $\sharp$-conjugations.
	
	\begin{example}
		Let $\mathscr{A}=\mathbb{C}\oplus\mathbb{C}$ and equip $\mathscr{A}$ with the $*$-conjugate-automorphism $(a,b)^\sharp=(\overline{b},\overline{a})$. Let $\mathscr{E}=\mathscr{A}$ be the standard right Hilbert $\mathscr{A}$-module, with 
		$$\langle x,y\rangle=x^*y=(\overline{x_1}y_1,\overline{x_2}y_2),$$
		where $x=(x_1,x_2)\in\mathscr{A}$ and $y=(y_1,y_2) \in\mathscr{A}$.
		
		We claim that the unitary operator $U:\mathscr{E}\to\mathscr{E}$ given by $$U(z_1,z_2)=(z_1,-z_2)$$ cannot be written as a product of two $\sharp$-conjugations. First, we determine all $\sharp$-conjugations on $\mathscr{E}$. 
		
		Put $e_1=(1,0)$ and $e_2=(0,1)$. Let $C$ be a $\sharp$-conjugation. Note that $\langle Ce_1,Ce_1\rangle=\langle e_1,e_1\rangle^\sharp=e_2$. If $Ce_1=(\gamma, \alpha)$, then $\langle (\gamma, \alpha), (\gamma, \alpha)\rangle=(|\gamma|^2, |\alpha|^2)=e_2$, whence $\gamma=0$. Thus, there is an $\alpha\in\mathbb{T}$ such that $Ce_1=(0,\alpha)$. Similarly, there is a $\beta\in\mathbb{T}$ such that $Ce_2=(\beta,0)$. Since $C$ is involutive and $\sharp$-conjugate-linear,
		\[
		e_1=C^2e_1=C(0,\alpha)=C(e_2\alpha)=Ce_2\,\overline{\alpha}=(\beta,0)\overline{\alpha}.
		\]
		Consequently, $\beta\overline{\alpha}=1$, and hence $\beta=\alpha$. Therefore every $\sharp$-conjugation has the form
		\[
		C_\alpha(z_1,z_2)=C_\alpha(z_1e_1+z_2e_2)=\overline{z_1}(0,\alpha)+\overline{z_2}(\alpha,0)
		=(\alpha\overline{z_2},\alpha\overline{z_1}),\qquad \alpha\in\mathbb{T}.
		\]
		Conversely, it is immediate that each $C_\alpha$ is a $\sharp$-conjugation.
		
		Now let $C_\alpha$ and $C_\beta$ be any two $\sharp$-conjugations. Then
		\[
		(C_\alpha C_\beta)(z_1,z_2)
		=
		C_\alpha(\beta\overline{z_2},\beta\overline{z_1})
		=
		(\alpha\overline{\beta}z_1,\alpha\overline{\beta}z_2).
		\]
		Hence $C_\alpha C_\beta=\alpha\overline{\beta}\,I_{\mathscr{E}}$. Thus every product of two $\sharp$-conjugations is a scalar multiple of the identity.
		
		On the other hand, $Ue_1=e_1$ and $Ue_2=-e_2$, so $U$ is not a scalar multiple of $I_{\mathscr{E}}$. Thus, $U$ is a unitary operator on $\mathscr{E}$ which is not a product of two $\sharp$-conjugations.
	\end{example}
	
	\section{Concluding remarks}\label{sec6}
	
	Extending the notion of conjugation to Hilbert $C^*$-modules requires finding a suitable counterpart of the identity 
	\[
	\langle Cx,y\rangle=\overline{\langle x,Cy\rangle} 
	\] 
	for conjugate-linear maps on a Hilbert $\mathscr{A}$-module $\mathscr{E}$. In this paper, we adopt one possible approach. Another apparently natural way is to define the adjoint of a bounded conjugate-linear operator $S: \mathscr{E}\to \mathscr{E}$ as a bounded conjugate-linear operator $S^*: \mathscr{E}\to \mathscr{E}$ satisfying 
	\[
	\langle Sx,y\rangle=\langle x,S^*y\rangle^*. 
	\] 
	This leads naturally to defining a conjugation on $\mathscr{E}$ as a conjugate-linear map $C:\mathscr{E}\to\mathscr{E}$ such that $C^2=I_\mathscr{E}$ and $\langle Cx,y\rangle=\langle x, Cy\rangle^*$ for all $x,y\in\mathscr{E}$.

	In this way, we expect the validity of the crucial formula
	\[
	C(xa)=C(x)a^* \qquad \text{for all } x\in\mathscr{E},\ a\in\mathscr{A}. 
	\] 
	as it represents the natural extension of conjugate-linearity to the setting of Hilbert $C^*$-modules. In this direction we have the next result. 
	
	Let $\mathscr{E}$ be Hilbert $\mathscr{A}$-module. We denote the linear span of the set $\{\langle x,y\rangle : x,y\in \mathscr{E}\}$ by $\langle \mathscr{E}, \mathscr{E} \rangle$. Then $\overline{\langle \mathscr{E}, \mathscr{E} \rangle}$ is a closed two-sided ideal in $\mathscr{A}$. In addition, $\mathscr{E}$ is called \emph{full} if $\overline{\langle \mathscr{E}, \mathscr{E} \rangle}=\mathscr{A}$.
	
	\begin{theorem}\label{th conj}
		Let $\mathscr{E}$ be a Hilbert $\mathscr{A}$-module admitting at least one conjugation in the sense of $\langle Cx,y\rangle=\langle x, Cy\rangle^*$. Then the following statements are equivalent.
		\begin{enumerate}
			\item For all conjugations $C$ on $\mathscr{E}$ it holds that 
			$$C(xa) = C(x)a^*\qquad(x\in\mathscr{E}, a\in\mathscr{A}).$$
			\item For some conjugation $C$ on $\mathscr{E}$ it holds that 
			$$C(xa) = C(x)a^*\qquad(x\in\mathscr{E}, a\in\mathscr{A}).$$
			\item For every $x,y\in \mathscr{E}$ and $a\in \mathscr{A}$, 
			$$
			\langle xa,y\rangle=\langle x,ya^*\rangle.
			$$
			\item The closed ideal $\overline{\langle \mathscr{E}, \mathscr{E} \rangle}$ is contained in the center $\mathcal{Z}(\mathscr{A})$ of $\mathscr{A}$.
		\end{enumerate}
	\end{theorem}
	
	\begin{proof}
		The implication $(1) \Rightarrow (2)$ is immediate.
		
		For $(2) \Rightarrow (3)$, assume that $C$ is a conjugation satisfying $C(xa)=C(x)a^*$. Then for arbitrary $x,y\in\mathscr{E}$ and $a\in\mathscr{A}$, using the properties of a conjugation, we obtain
		$$
		\langle xa,y\rangle=\langle C(y), C(xa)\rangle= \langle C(y), C(x)a^*\rangle=\langle C(y), C(x)\rangle a^*=\langle x ,y\rangle a^*=\langle x,ya^*\rangle.
		$$
		Thus, (3) holds.
		
		For $(3) \Rightarrow (4)$, suppose (3) is valid. Then for every $a\in\mathscr{A}$ and $x,y\in\mathscr{E}$, we have
		\begin{align*}
			a\langle x,y\rangle&= \bigl(\langle y,x\rangle a^*\bigr)^*=\langle y,xa^*\rangle^*=\langle xa^*,y\rangle = \langle x,ya\rangle=\langle x,y\rangle a.
		\end{align*}
		Since $\overline{\langle \mathscr{E}, \mathscr{E} \rangle}$ is the closed linear span of all inner products $\langle x,y\rangle$, it follows that $\overline{\langle \mathscr{E}, \mathscr{E} \rangle}\subseteq \mathcal{Z}(\mathscr{A})$.
		
		For $(4) \Rightarrow (1)$, let $C$ be an arbitrary conjugation on $\mathscr{E}$. Take $x,y\in\mathscr{E}$ and $a\in\mathscr{A}$. Then
		$$
		\langle C(xa), y\rangle=\langle C(y), xa\rangle=\langle C(y), x\rangle a=a\langle C(y), x\rangle = a\langle C(x), y\rangle = \langle C(x)a^*, y\rangle,
		$$
		Since $y\in\mathscr{E}$ was arbitrary, the nondegeneracy of the inner product yields $C(xa)=C(x)a^*$. As $C$ was arbitrary, (1) follows.
	\end{proof}
	
	As a consequence of this theorem, we obtain the following result, which shows that the condition 
	\[
	C(xa)=C(x)a^* \qquad (x\in\mathscr{E},\ a\in\mathscr{A}) 
	\] 
	essentially forces the $C^*$-algebra $\mathscr{A}$ to be commutative.
	
	\begin{corollary}
		Let $\mathscr{E}$ be a Hilbert $\mathscr{A}$-module admitting a conjugation $C$ in the sense of $\langle Cx,y\rangle=\langle x,Cy\rangle^*$. If $\mathscr{A}$ is commutative, then $C(xa)=C(x)a^*$ for every $x\in \mathscr{E}$ and $a\in \mathscr{A}$. Conversely, if $\mathscr{E}$ is full and the above equality holds, then $\mathscr{A}$ is commutative.
	\end{corollary}
	
	If we naturally consider a Hilbert space $\mathscr{H}$ as a right Hilbert $C^*$-module over $\mathbb{B}(\mathscr{H})$, then there exists no conjugation as shown below.
	
	\begin{example} \label{Esk11}
		Let $\mathscr{H}$ be a Hilbert space, with inner product linear in the first variable and conjugate-linear in the second, as usual. Let $\overline{\mathscr{H}}$ denote the conjugate Hilbert space of $\mathscr{H}$. Its elements are written $\overline{x}$, where $x\in\mathscr{H}$. We set $\overline{x}+\overline{y}:=\overline{x+y}$ and $\lambda\overline{x}:=\overline{\overline{\lambda}x},\,\,\lambda\in\mathbb{C}$. Define the right action of $\mathbb{B}(\mathscr{H})$ by $\overline{x}\cdot T:=\overline{T^*x},\,\,T\in\mathbb{B}(\mathscr{H})$, and
		$$\langle \overline{x},\overline{y}\rangle_1:=x\overline{\otimes} y.$$
		Then
		$\langle \overline{x},\overline{y}\rangle_1^*=(y\overline{\otimes} x)=\langle \overline{y},\overline{x}\rangle_1$, and
		$\langle \overline{x},\overline{x}\rangle_1=x\overline{\otimes} x\geq0$. Moreover, $\langle\overline{x},\overline{x}\rangle_1=0$ implies that $x=0$. 
		The module compatibility relations are satisfied, since
		$$\langle \overline{x}\cdot T,\overline{y}\rangle_1=\langle \overline{T^*x},\overline{y}\rangle_1=T^*x\overline{\otimes} y=
		T^*(x\overline{\otimes} y),$$ and $$\langle \overline{x},\overline{y}\cdot T\rangle_1=\langle\overline{x},\overline{T^*y}\rangle_1=x\overline{\otimes} T^*y=(x\overline{\otimes} y)T.$$
		Moreover, $\|\overline{x}\|_1^2=\|x\overline{\otimes} x\|=\|x\|^2$, since $(x \overline{\otimes} x)x=\langle x,x\rangle x=\|x\|^2x$.
		
		Therefore, $\overline{\mathscr{H}}$ is complete in the induced norm and is a Hilbert $\mathbb{B}(\mathscr{H})$-module.
		
		Assume that a conjugation $C:\overline{\mathscr{H}}\to\overline{\mathscr{H}}$ satisfying $\langle Sx,y\rangle=\langle x,S^*y\rangle^*$ exists. Choose a unit vector $\overline{x}\in\overline{\mathscr{H}}$. For every $\overline{y}\in\overline{\mathscr{H}}$, we have
		$\langle C\overline{x},C\overline{y}\rangle_1=\langle \overline{y},\overline{x}\rangle_1$. Write $C\overline{x}=\overline{u}$ and $C\overline{y}=\overline{v}$ for some $u,v\in\mathscr{H}$. Then, for every $z\in\mathscr{H}$,
		$$
		(u\overline{\otimes} v)z
		=
		\langle C\overline{x},C\overline{y}\rangle_1z
		=
		\langle \overline{y},\overline{x}\rangle_1z
		=
		(y\overline{\otimes} x)z.
		$$
		Hence, $u\overline{\otimes} v=y\overline{\otimes} x$. Consequently, using the definition of the rank-one operator,
		$$
		\langle z, v\rangle u = \langle z, x\rangle y
		\qquad (z\in\mathscr{H}).
		$$
		Taking $z=x$ and using $\|x\|=1$, we obtain $\langle x, v\rangle u = y$. For the fixed vector $\overline{x}$, the vector $u$ is fixed. Hence, every $y\in\mathscr{H}$ belongs to the one-dimensional subspace $\mathbb{C} u$. Therefore, $\dim\mathscr{H}\leq1$, contrary to the hypothesis $\dim\mathscr{H}\geq2$. Thus, $\overline{\mathscr{H}}$ has no conjugation.
	\end{example}

	Moreover, even when we consider $\mathscr{A}$ as a Hilbert $C^*$-module over itself, we arrive at the following conclusion.
	
	\begin{theorem}\label{th conj2}
		Let $\mathscr{A}$ be a unital $C^*$-algebra, equipped with its standard Hilbert $\mathscr{A}$-module structure given by $\langle x,y\rangle=x^*y$ for $x,y\in \mathscr{A}$. Let $T\in\mathbb{L}_c(\mathscr{A})$ in the sense that $\langle Tx,y\rangle=\langle x,T^*y\rangle^*$ for some map $T^*$. Then the range of $T$ satisfies $$\mathcal{R}(T)\subseteq \mathcal{Z}(\mathscr{A}),$$
		and $T$ has the form $$T(x)=x^*a$$ for all $x\in A$, where $a\in \mathcal{Z}(\mathscr{A})\cap{\rm Ann}(\mathcal{I})$, with $\mathcal{I}=\overline{\sspan}\{u[x,y]v:u,v,x,y\in \mathscr{A}\}$ is a closed two-sided ideal of $\mathscr{A}$. Moreover, $T^*=T$. 
		
		In particular, if $\mathscr{A}$ is a noncommutative $C^*$-algebra with $\mathcal{Z}(\mathscr{A})=\mathbb{C}$, then $T=0$.
	\end{theorem}
	
	\begin{proof}
		Set $a=T1$. Since $T\in\mathbb{L}_c(\mathscr{A})$, there exists $T^*\in\mathbb{L}_c(\mathscr{A})$ such that $\langle Tx,y\rangle=\langle T^*y,x\rangle$ for all $x,y\in \mathscr{A}$. Taking $x=1$ gives $\langle T1,y\rangle=\langle T^*y,1\rangle$, hence $a^*y=(T^*y)^*$, so 
		\begin{align}\label{(1)}
			T^*y=y^*a, \qquad y\in \mathscr{A}.
		\end{align}
		Now for arbitrary $x,y\in \mathscr{A}$, we have $$\langle Tx,y\rangle=\langle T^*y,x\rangle=(T^*y)^*x=a^*yx.$$ Since $\langle Tx,y\rangle=(Tx)^*y$, it follows that $(Tx)^*=a^*x$, and therefore 
		\begin{align}\label{(2)}
			Tx=x^*a, \qquad x\in \mathscr{A}.
		\end{align}
		Also, from the equality $a^*yx=(Tx)^*y$ and using \eqref{(2)}, we get $a^*xy=a^*yx$, so $a^*[x,y]=0$ for all $x,y\in \mathscr{A}$. Taking adjoints yields $[x,y]a=0$ for all $x,y\in \mathscr{A}$.
		
		Let $$\mathcal{I}=\overline{\sspan}\{u[x,y]v:u,v,x,y\in \mathscr{A}\}.$$
		Using the identity $[x,y]z=[x,yz]-y[x,z]$, we obtain $[x,y]va=[x,yv]a-y[x,v]a=0$, hence $u[x,y]va=0$ for all $u,v,x,y\in \mathscr{A}$. Therefore, $\mathcal{I}a=0$. Taking adjoints gives $a^*\mathcal{I}=0$.
		
		We next prove that $a\mathcal{I}=0$. Let $z\in \mathcal{I}$. Since $\mathcal{I}$ is a two-sided ideal, $az\in \mathcal{I}$. By $a^*\mathcal{I}=0$, we have $a^*(az)=0$. Then $\|az\|^2=\|(az)^*(az)\|=\|z^*a^*az\|=0$, so $az=0$. Thus, $a\mathcal{I}=0$. Combining with $\mathcal{I}a=0$ yields $a\mathcal{I}=\mathcal{I}a=0$.
		
		Since $\mathscr{A}/\mathcal{I}$ is commutative, we have $[a,x]\in \mathcal{I}$ for all $x\in \mathscr{A}$. Let $J=\{z\in \mathscr{A}:z\mathcal{I}=\mathcal{I}z=0\}$ be the annihilator of $\mathcal{I}$. Then from $[a,x]\in \mathcal{I}$ and $a\mathcal{I}=\mathcal{I}a=0$, it follows that $[a,x]\in J$ as well. Hence, $[a,x]\in \mathcal{I}\cap J=\{0\}$, so $[a,x]=0$ for all $x\in \mathscr{A}$. Thus, $a\in \mathcal{Z}(\mathscr{A})$.
		
		Now using $a\mathcal{I}=\mathcal{I}a=0$ and $a\in \mathcal{Z}(\mathscr{A})$, we compute $[T(x),y]=[x^*a,y]=a[x^*,y]=0$ for all $x,y\in \mathscr{A}$. Hence, $\mathcal{R}(T)\subseteq \mathcal{Z}(\mathscr{A})$. Also, \eqref{(1)} and \eqref{(2)} we immediately get $T^*=T$.
		
		Finally, suppose $\mathscr{A}$ is noncommutative with $\mathcal{Z}(\mathscr{A})=\mathbb{C}$. Then from $T(x)=x^*a$ and $a\in \mathcal{Z}(\mathscr{A})=\mathbb{C}$, we have $T(x)=\lambda x^*$ for some $\lambda\in\mathbb{C}$. Choose $x,y\in \mathscr{A}$ with $xy\neq yx$. From $\langle Tx,y\rangle=\langle T^*y,x\rangle$ and $T^*=T$, we get $\langle \lambda x^*,y\rangle=\langle \lambda y^*,x\rangle$, which simplifies to $\bar{\lambda}xy=\bar{\lambda}yx$. Since $xy\neq yx$, this forces $\lambda=0$. Therefore, $T=0$.
	\end{proof}
	
	Thus, not only do we obtain the rather surprising statement that every operator $T\in\mathbb{L}_c(\mathscr{A})$ is self-adjoint, but also, even for the simple noncommutative finite-dimensional $C^*$-algebra $\mathbb{M}_n(\mathbb{C})$ of all $n \times n$ complex matrices, there exists no nonzero conjugate-linear adjointable operator, and hence, a fortiori, no conjugation at all. 
	
	The following example shows that, contrary to the case discussed above, this situation does not arise under the approach adopted in this paper.
	
	\begin{example}\label{th conj3}
		Define the map
		\[
		\sharp:\mathbb{M}_n(\mathbb{C})\to \mathbb{M}_n(\mathbb{C}),\qquad (a_{ij})^\sharp=(\bar{a_{ij}})
		\]
		It is seen that $\sharp$ is a $*$-conjugate-automorphism on $\mathbb{M}_n(\mathbb{C})$. Let $\mathscr{E}=\mathbb{M}_n(\mathbb{C})$ be the Hilbert $C^*$-module over $\mathbb{M}_n(\mathbb{C})$ via the inner product $\langle A,B\rangle=A^*B$ for $A,B\in \mathbb{M}_n(\mathbb{C})$. Set
		\[
		C: \mathbb{M}_n(\mathbb{C}) \to \mathbb{M}_n(\mathbb{C}), \qquad C(A) = \overline{A},
		\]
		where $\overline{A}$ is the entrywise complex conjugate of $A$. Then, it is easy to verify that $C$ is a conjugation (in the sense of Definition \ref{ours}) on $\mathscr{E}$.
	\end{example}
	
	In the case where the $C^*$-algebra $\mathscr{A}$ is commutative, the $*$-operation satisfies all the conditions of Definition \ref{Esk0}, so in this case, one can define the conjugation in terms of the $*$-operation as $\langle Cx,y\rangle =\langle x,Cy\rangle^*$.
	
	
	\medskip
	\noindent \textit{Author Contributions Statement.} All authors wrote, edited, and reviewed the manuscript.
	
	\medskip
	\noindent \textit{Conflict of Interest Statement.} On behalf of the authors, the corresponding author states that there is no conflict of interest.
	
	\medskip
	\noindent\textit{Data Availability Statement.} Data sharing is not applicable to this article as no datasets were generated or analyzed during the current study.
	
	\medskip
	\noindent \textit{Funding Declaration.} This research received no funding.

	\medskip

\end{document}